\documentclass[onefignum,onetabnum]{siamart171218}

\usepackage{latexsym}
\usepackage{amsmath}
\usepackage{amssymb}
\usepackage{amsfonts}
\usepackage{graphics,graphicx}
\usepackage{comment}
\usepackage{xcolor}
\usepackage{url}

\usepackage{tikz}
\usetikzlibrary{calc,patterns,decorations.pathmorphing,positioning}

\usepackage{subcaption}

\def\claim#1{\begin{trivlist}\item[\hskip\labelsep\bf#1]\it}
\def\endclaim{\end{trivlist}}

\numberwithin{equation}{section}

\newtheorem{pr}[theorem]{Proposition}
\newtheorem{remark}[theorem]{Remark}
\newtheorem{assumption}[theorem]{Assumption}
\newcommand{\eproof}{{\mbox{\ }~\hfill
\mbox{\large $\Box$} \par \vskip 10pt}}
\newcommand{\pf}{\noindent{\bf Proof}}

\title{Resolvent Estimates for Viscoelastic Systems of Extended Maxwell Type and their Applications}

\author{Maarten V. de Hoop\thanks{Simons Chair in Computational and Applied Mathematics and Earth Science, Rice University, Houston TX, USA (mdehoop@rice.edu). Supported by the Simons Foundation under the MATH + X program, the National Science Foundation under grant DMS-2108175, and the corporate members of the Geo-Mathematical Imaging Group at Rice University.} \and
  Masato Kimura\thanks{Faculty of Mathematics and Physics, Kanazawa University, Kanazawa 920-1192, Japan (mkimura@se.kanazawa-u.ac.jp).} \and
  Ching-Lung Lin\thanks{Department of Mathematics, National Cheng-Kung University, Tainan 701, Taiwan (cllin2@mail.ncku.edu.tw). Partially supported by the Ministry of Science and Technology of Taiwan.} \and
  Gen Nakamura\thanks{Department of Mathematics, Hokkaido University, Sapporo 060-0808, Japan and Research Institute of Electronic Science, Hokkaido University, Sapporo 060-0812, Japan (gnaka@math.sci.hokudai.ac.jp). Partially supported by grant-in-aid for Scientific Research (22K03366) of the Japan Society for the Promotion of Science.}}

\date{}

\begin{document}

\maketitle

\begin{abstract}
In the theory of viscoelasticity, an important class of models admits a representation in terms of springs and dashpots. Widely used members of this class are the Maxwell model and its extended version. This paper concerns resolvent estimates for the system of equations for the anisotropic, extended Maxwell model, abbreviated as the EMM, and its marginal realization which includes an inertia term; special attention is paid to the introduction of augmented variables. This leads to the augmented system that will also be referred to as the ``original'' system. A reduced system is then formed which encodes essentially the EMM; it is a closed system with respect to the particle velocity and the difference between the elastic and viscous strains. Based on resolvent estimates, it is shown that the original and reduced systems generate $C_0$-groups and the reduced system generates a $C_0$-semigroup of contraction. Naturally, the EMM can be written in integrodifferential form leading explicitly to relaxation and a viscoelastic integro-differential system. However, there is a difference between the original and integrodifferential systems, in general, with consequences for whether their solutions generate semigroups or not. Finally, an energy estimate is obtained for the reduced system, and it is proven that its solutions decay exponentially as time tends to infinity. The limiting amplitude principle follows readily from these two results.
\end{abstract}



\section{Introduction}\label{sec1}

Let $\Omega\subset\mathbb{R}^{d}$ with $d=2,3$ be a bounded domain on which a spring-dashpot model of a viscoelastic medium is defined. We assume that its boundary $\partial\Omega$ is connected and Lipschitz smooth. We divide $\partial\Omega$ into $\partial\Omega = \overline{\Gamma_D}\cup\overline{\Gamma_N}$, where $\Gamma_D,\,\Gamma_N\subset\partial\Omega$ are connected open sets and we assume that $\Gamma_D\not=\emptyset$, $\Gamma_D\cap\Gamma_N=\emptyset$ and if $d=3$, then their boundaries $\partial\Gamma_D,\,\partial\Gamma_N$ are Lipschitz smooth. We emphasize that the setup with $\partial\Omega$, $\Gamma_D$, $\Gamma_N$ underpins the consideration of the so-called mixed type boundary condition. Our analysis extends to the case where $\partial\Omega$ consists of several connected components and $\Gamma_D$, $\Gamma_N$ are unions of these components.

Let $x\in\Omega$ be a point in space and $t\in{\mathbb R}$ be time. For each $1\le j\le n$ with a fixed $n\in{\mathbb N}$, let $C_j=C_j(x)$ be a stiffness tensor and $\phi_j=\phi_j(x,t)$  be a tensor describing the effect of viscosity; these are rank $4$ and rank $2$ tensors, respectively. Further, let $\rho=\rho(x)$ be the density defined on $\Omega$ and each $\eta_j=\eta_j(x)$, $1\le j\le n$ be the viscosity of the $j$-th dashpot, respectively.

\medskip
Throughout this paper, the assumptions for $C_j$, $\eta_j$ and $\rho$ are as follows.

\begin{assumption}\label{assumption}${}$
\newline
{\rm(i)}
$C_j,\,\eta_j,\,\,\rho\in L^\infty(\Omega)$.

\medskip
\noindent
{\rm (ii)} (full symmetry) $(C_j)_{klrs}=(C_j)_{rskl }=(C_j)_{klsr},\,\,j\le k,l,r,s\le d$\,\,\text{a.e. in $\Omega$}.

\medskip
\noindent
{\rm (iii)} (strong convexity) There
exists a constant $\alpha_0>0$ such that for any $d\times d$ symmetric matrix $w=(w_{kl})$
\begin{equation}\label{1.4}
(C_jw)w\ge\alpha_0|w|^2,\,\,
\,\,\text{a.e. in $\Omega$},
\end{equation}
where the contraction $C_j w$ of the rank  $4$ tensor $C_j$ with  rank $2$ tensor $w$, and the contraction $(C_j w)w$ of the rank  $2$ tensors $C_j w$ and $w$ are defined as $(C_jw)w=\sum_{k,l=1}^d \left(\sum_{r,s=1}^d(C_j)_{klrs}w_{rs}\right)w_{kl}$. 

\medskip
\noindent
{\rm(iv)} There exist $\beta_0>0$ and $\gamma_0>0$ such that
$$
\eta_j\ge \beta_0,\,1\le j\le n\,\,\text{and}\,\,
\rho\ge \gamma_0
$$
\text{ a.e. in $\Omega$}. 
\end{assumption}

\begin{figure}[h]  
	\centering 











\def\myscale{1.2} 
		\begin{tikzpicture}[scale=\myscale,transform shape,every node/.style={outer sep=0pt},thick,
				spring/.style={thick,decorate,decoration={zigzag,pre length=6*\myscale,
				post length=6*\myscale,segment length=14*\myscale,amplitude=10*\myscale}},
				dampic/.pic={\fill[white] (-0.1,-0.3) rectangle (0.3,0.3);
					\draw (-0.3,0.3) -| (0.3,-0.3) -- (-0.3,-0.3);
				\draw[line width=1mm] (-0.1,-0.3) -- (-0.1,0.3);}
			]
			\foreach \i\j\k\l in {01/1/0/3cm,02/2/-1.5/3cm,99/n/-3.5/3cm}
			{
				\node[coordinate] 
				at (0,\k)
				(localcenter\i) {};

				\node (rect) [rectangle,draw,dotted,minimum width=4cm,minimum height=1.2cm,
				] at (localcenter\i) {};
				\node [label={[label distance=1.8cm]5:$M_\j$}] at (localcenter\i) {};

				\node[coordinate,
				right=\l of localcenter\i] (localright\i) {};

				\node[coordinate,
				left=2cm of localcenter\i] (localmidleft\i) {};

				\node[coordinate,
				left=\l of localcenter\i] (localleft\i) {};

				\draw[thick] (localmidleft\i) to (localleft\i);

				\draw[spring] ([yshift=0mm]localmidleft\i.east) coordinate(aux)
				-- (localcenter\i.west|-aux) node[midway,above=1mm]{};

				\draw ([yshift=-0mm]localcenter\i.east) coordinate(aux')
				-- 
				(localright\i.west|-aux') pic[midway]{dampic} node[midway,below=3mm]{};
			}
			\draw [shorten >=0.8cm,shorten <=0.8cm,loosely dotted,thick] (localcenter02) to (localcenter99);
			\draw [shorten >=0.8cm,shorten <=0.8cm,loosely dotted,thick] (localright02) to (localright99);
			\draw [shorten >=0.8cm,shorten <=0.8cm,loosely dotted,thick] (localleft02) to (localleft99);
			\draw [shorten >=0.0cm,shorten <=-0.7cm,thick] (localleft02) to (localleft01);
			\draw [shorten >=0.0cm,shorten <=-0.7cm,thick] (localright02) to (localright01);
			\draw [shorten >=0.0cm,shorten <=1.3cm,thick] (localleft02) to (localleft99);
			\draw [shorten >=0.0cm,shorten <=1.3cm,thick] (localright02) to (localright99);
			\draw[thick] ($(localright02)!-0.2!(localcenter02)$) to (localright02);
			\draw[thick] ($(localleft02)!-0.2!(localcenter02)$) to (localleft02);
		\end{tikzpicture}%
	\caption{Extended Maxwell model and its one unit called the Maxwell model, where the zigzag and piston describe a spring and dashpot, respectively. Here, $M_j$ denotes the $j$-th Maxwell constituent. }\label{fig:EMM}
\end{figure}
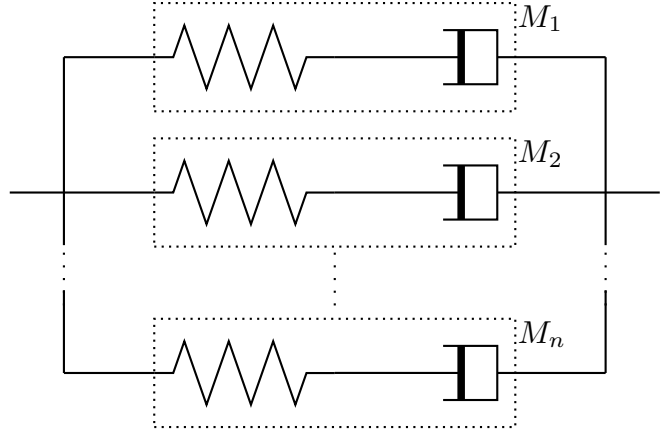  

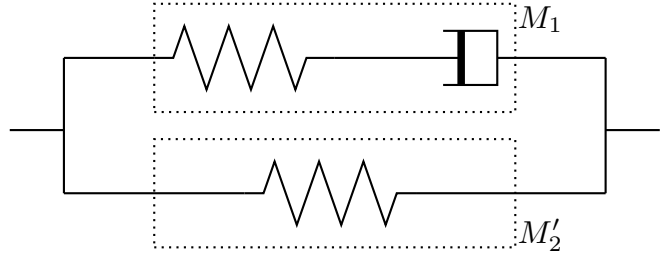
\begin{figure}[h]  
	\centering 
	\def\myscale{1.2}
	\begin{tikzpicture}[scale=\myscale,transform shape,every node/.style={outer sep=0pt},thick,
			spring/.style={thick,decorate,decoration={zigzag,pre length=6*\myscale,
			post length=6*\myscale,segment length=14*\myscale,amplitude=10*\myscale}},
			dampic/.pic={\fill[white] (-0.1,-0.3) rectangle (0.3,0.3);
				\draw (-0.3,0.3) -| (0.3,-0.3) -- (-0.3,-0.3);
			\draw[line width=1mm] (-0.1,-0.3) -- (-0.1,0.3);}
		]
		\foreach \i\j\k\l in {01/1/0/3cm,05/XX/-0.8/3cm,09/2/-1.5/3cm}
		{
			\node[coordinate] 
			at (0,\k)
			(localcenter\i) {};

			\node[coordinate,
			right=\l of localcenter\i] (localright\i) {};

			\node[coordinate,
			left=2cm of localcenter\i] (localmidleft\i) {};

			\node[coordinate,
			left=\l of localcenter\i] (localleft\i) {};

		}

		\node (rect) [rectangle,draw,dotted,minimum width=4cm,minimum height=1.2cm,
		] at (localcenter01) {};
		\node [label={[label distance=1.8cm]5:$M_1$}] at (localcenter01) {};

		\draw[thick] (localmidleft01) to (localleft01);

		\draw[spring] ([yshift=0mm]localmidleft01.east) coordinate(aux)
		-- (localcenter01.west|-aux) node[midway,above=1mm]{};

		\draw ([yshift=-0mm]localcenter01.east) coordinate(aux')
		-- 
		(localright01.west|-aux') pic[midway]{dampic} node[midway,below=3mm]{};

		\node (rect) [rectangle,draw,dotted,minimum width=4cm,minimum height=1.2cm,
		] at (localcenter09) {};
		\node [label={[label distance=1.8cm]-5:$M_2^\prime$}] at (localcenter09) {};

		\draw[thick] ($(localmidleft09)!0.5!(localcenter09)$) to (localleft09);

		\draw[spring] ($(localmidleft09)!0.5!(localcenter09)$) coordinate(aux)
		-- ($(localright09)!0.7!(localcenter09)$) node[midway,above=1mm]{};

		\draw[thick] ($(localright09)!0.7!(localcenter09)$) to (localright09);

		\draw[thick] ($(localright05)!-0.2!(localcenter05)$) to (localright05);
		\draw[thick] ($(localleft05)!-0.2!(localcenter05)$) to (localleft05);

		\draw [shorten >=0.0cm,shorten <=-0.0cm,thick] (localleft09) to (localleft01);
		\draw [shorten >=0.0cm,shorten <=-0.0cm,thick] (localright09) to (localright01);
	\end{tikzpicture}%
	\caption{Standard linear solid model} \label{fig:SLSM}  
\end{figure}  

\medskip

Historically, viscoelasticity is introduced through relaxation leading to systems of integrodifferential equations, which we refer to as VID systems \cite{C, Lakes}. In the case of special, parametric models, representable by springs and dashpots \cite{Lakes}, augmented variables can be introduced to cast the systems of integrodifferential equations into systems of differential equations. This alternative mathematical description affects the meaning of initial values. This also affects whether the solutions form a semigroup. Here, we will present the analysis for, and clarify the properties of systems associated with the so-called Extended Maxwell Model (EMM), see Figure~\ref{fig:EMM}. This model has been commonly used in geophysics \cite{Peltier_1974}, for example. To accommodate general anisotropy, we introduce its tensor form. We will also consider the related Extended Standard Linear Solid Model (ESLSM), see Figure~\ref{fig:SLSM} for an illustration of one unit. In the case of the EMM, we will prove that the solutions of the integrodifferential system (ID) do not generate a semigroup, but that the solutions of the augmented system of differential equations (AD) do generate not only a semigroup but also a group. The integrodifferential system generates exponentially decaying solutions; exponentially decaying solutions are generated by the AD only upon a reduction of the system eliminating quasi-static modes. In fact, the proofs make use of distinct energy functions. We will present the proof for the ID system in a companion paper \cite{IDUD_2023}. In various remarks, we will indicate which results apply to the ESLSM.

As a direct application of the result on exponential decay in time, we show that the solutions of the reduced system satisfy the limiting amplitude principle. That is, if a time-harmonic vibration is given on a part of the boundary, this principle implies how fast solutions of this system converge to time-harmonic solutions. We note that the time-harmonic vibration given on the boundary has a transitional period of time before becoming a time-harmonic vibration. If this convergence is fast, one can quickly generate many time-harmonic solutions by switching the frequency of the time-harmonic vibration given on a part of the boundary. This principle has been accepted without mathematical proof in applications. In this paper, we give a mathematical justification of this principle. For the Kelvin-Voigt model, which can be directly described as a system of differential equations, the exponential decay of solutions was proven before. 
By using the inner product introduced in \cite{Potier}, it follows that this system generates a holomorphic semigroup and has the mentioned decay of solutions. Based on these results the limiting amplitude was proven when the medium is isotropic in \cite{Higashimori} and \cite{Jiang} for the one space dimensional case and the three space dimensional case, respectively. Applications of the limiting amplitude principle appear, for example, in magnetic resonance elastography (MRE) \cite{Muthupillai} and in exploration seismology with vibroseis. In MRE, one uses a special pulse echo sequence of MRI, and measures the time-harmonic quasi-shear wave inside a tissue of an organ generated by a time-harmonic vibration at its surface; this was analyzed by Papazoglou \cite{Papazoglou} using a ESLSM.

The remainder of the paper is organized as follows. In Section 2, we introduce the AD system and carry out a preliminary study of the associated system in the Laplace domain which we refer to as the $\lambda-$system or equation. In Section 3, we analyze the relation between the ID system and the AD system and what causes their difference. We show the existence of a $C_0$-group for the AD system in Section 5 by deriving the Hille-Yoshida type estimate for the resolvent in Section 4. To show the limiting amplitude principle for the AD system, we need to have at least the following two conditions to be satisfied. They are that the resolvent set contains the imaginary axis of the complex plane and that solutions of the AD system have a uniform decaying property, for example, polynomial or exponential as $t \rightarrow \infty$. However, these conditions do not hold for the AD system due to the existence of a stationary solution. Hence, we introduce a reduced system (RD), and then prove that it generates a $C_0$-group of contraction in Section 6 and that it has the exponentially decaying property of solutions in Section 7; thus this system satisfies the mentioned two conditions. We conclude with establishing the limiting amplitude principle for the reduced system.

\section{AD system} 

Here, we introduce the AD system for the EMM. In the absence of an exterior force, the vibration with small deformation of a viscoelastic medium, on $\Omega$, modeled as the AD system is expressed in terms of the elastic displacement, $u = u(x,t)$, and viscous strains, $\phi = \phi(x,t)$, as follows,
\begin{equation}\label{2.1}
\left\{
\begin{array}{ll}
\rho\partial_t^2u-{\rm div} \, \sigma[u,\phi]=0,\quad\sigma[u,\phi]=\sum_{j=1}^n\sigma_j[u,\phi_j],\\
\eta_j\partial_t\phi_j-\sigma_j[u,\phi_j]=0,\quad j=1,\ldots,n,\\[0.25cm]
u=0\quad\text{on}\quad\Gamma_D,\ |\Gamma_D|>0,\\
\sigma[u,\phi] \nu=0\quad\text{on}\quad\Gamma_N,\ |\Gamma_N|>0,\\[0.25cm]
(u,\partial_t u,\phi)|_{t=0}=(u^0,v^0,\phi^0)\quad\text{on}\quad\Omega,
\end{array}
\right.
\end{equation}
requiring compatibility between the initial and boundary values, where $\nu$ is the outward unit normal of $\partial\Omega$, and 
\begin{equation}
\sigma_j[u,\phi_j]:=C_j(e[u]-\phi_j),\quad
e[u]=\tfrac{1}{2}\{\nabla u+(\nabla u)^t\};
\end{equation}
we used the notation, $\phi=(\phi_1,\cdots,\phi_n)$. 

The AD system \eqref{2.1} and its energy dissipation structure were studied in \cite{Yamamoto2019} and in \cite{JTLi2023}, based on a preceding work \cite{KNTY2019} where a simplified AD system without the inertia term and its energy dissipation structure were considered.

The first two equations of \eqref{2.1} are considered on $\Omega \times {\mathbb R}$ unless otherwise specified in the further analysis. The function spaces for the solution and the initial data will be specified later, in Sections 4 and 5, while addressing the unique solvability of \eqref{2.1}; this will follow from the existence of a $C_0$-(semi)group for \eqref{2.1}.

We rewrite the AD system by introducing the following notation. We let \newline $\breve{C} := {\rm block \, diag} (C_1, \cdots, C_n)$, $\breve{\eta} := {\rm block \, diag} (\eta_1 I_d, \cdots, \eta_n I_d)$ where $I_d$ stands for the $d \times d$ identity matrix, $\breve{I} :=I_{nd}$, and $O$ denote the zero matrix. Then we have
\begin{equation}\label{2.2}
\sigma[u,\phi]
=\mathfrak{T}\{\breve{C}(e[u]\breve{I}-\phi)\},
\end{equation}
where $\mathfrak{T}$ denotes the trace for the diagonal blocks, $\breve{C}\phi := {\rm block \, diag} (C_1\phi_1, \cdots, C_n\phi_n)$, $C_j\phi_i := (C_j\phi_i)_{kl}$, $(C_j\phi_i)_{kl} = \sum_{r,s=1}^n(C_j)_{klrs}(\phi_i)_{rs}$, and \\ $e[u] \breve{I} := {\rm block \, diag} (e[u], \cdots, e[u])$, while
$\breve{C}(e[u]\breve{I}) := {\rm block \, diag} (C_1e[u], \cdots, C_ne[u])$.

\begin{remark}\label{ESLS1} We interpret the system for the ESLSM as a special case of the system for the EMM in the sense that $\eta_j^{-1} C_j = 0$ for some of the values of $j$, and omitting the $\phi_j$ for such values.
\end{remark}

\medskip
In order to write the equations in \eqref{2.1} as a first order system with respect to the time derivative, we let $u_1=u$, $u_2=\partial_t u$ and $U:=(u_1,u_2,\phi)^t$. Then, using \eqref{2.1} and \eqref{2.2}, we get
\begin{equation}\label{2.3}
\left\{
\begin{aligned}
\partial_tu_1&=u_2,\\
\partial_tu_2&=\rho^{-1}{\rm div} \, \sigma[u,\phi]=\rho^{-1}{\rm div}\,\mathfrak{T}\{\breve{C}(e[u]\breve{I}-\phi)\},\\
\partial_t\phi&=\breve{\eta}^{-1}\{\breve{C}(e[u]\breve{I}-\phi)\},
\end{aligned}
\right.
\end{equation}
which takes the form, $\partial_t U=AU$, with
\begin{equation}\label{2.4}
A=
\begin{pmatrix}
O & I_d & O \\
\rho^{-1}{\rm div}\,\mathfrak{T}\{\breve{C}(e[\cdot]\breve{I})\}& O & -\rho^{-1}{\rm div}\,\mathfrak{T}\{\breve{C}\cdot\}\\
\breve{\eta}^{-1}\{\breve{C}(e[\cdot]\breve{I})\}& O & -\breve{\eta}^{-1}\{\breve{C}\cdot\}
\end{pmatrix}.
\end{equation}
For $\lambda>0$, we consider the ``$\lambda$-equation'' given as $(\lambda I-A)U=F:=(f_1,f_2,\omega)^t$. Writing out\footnote{In the next section, we introduce the notation $\hat{u}_1$, $\hat{u}_2$ and $\hat{\phi}_j$. We omit the $\hat{.}$, here, following standard notational convention for $\lambda$-equations.} the components, we obtain
\begin{equation}\label{2.6}
\left\{
\begin{aligned}
\lambda u_1&=u_2+f_1,\\
\lambda u_2&=\rho^{-1}{\rm div}\,\mathfrak{T}\{\breve{C}(e[u_1]\breve{I} - \phi)\} + f_2,\\
\lambda \phi&=\breve{\eta}^{-1}\{\breve{C}(e[u_1]\breve{I} - \phi)\} + \omega.
\end{aligned}
\right.
\end{equation}
In the above, we have omitted writing the boundary condition for conciseness. This convention will be used from now on for any $\lambda-$equation. We note that $\mathfrak{T}\{\breve{C}(e[u]\breve{I}-\phi)\}=\sigma[u,\phi]$ in  \eqref{2.2}
is the most important term which we have to analyze because it is directly connected to the Neumann boundary condition $\mathfrak{T}\{\breve{C}(e[u_1]\breve{I}-\phi)\}\nu=0$ on $\Gamma_N$ given in \eqref{2.1}. Next, we rewrite this Neumann boundary condition in terms of the equation for $u_1$.

Using the third equation of \eqref{2.6}, we find that
\begin{equation}\label{2.8}
\phi=(\lambda I+\breve{\eta}^{-1}\breve{C})^{-1}[\breve{\eta}^{-1}\{\breve{C}(e[u_1]\breve{I})\}+\omega].
\end{equation}
Combining \eqref{2.6} and \eqref{2.8}, we obtain
\begin{equation}\label{2.9}
\begin{aligned}
\lambda^2 u_1 =& \lambda u_2+\lambda f_1\\
=&\rho^{-1}{\rm div}\,\mathfrak{T}\{\breve{C}(e[u_1]\breve{I})\}-\rho^{-1}{\rm div}\,\mathfrak{T}\{\breve{C}\phi\}+f_2+\lambda f_1\\
=&\rho^{-1}{\rm div}\,\mathfrak{T}\{\breve{C}(e[u_1]\breve{I})\}+f_2+\lambda f_1\\
&-\rho^{-1}{\rm div}\,\mathfrak{T}\{\breve{C}(\lambda I+\breve{\eta}^{-1}\breve{C})^{-1}[\breve{\eta}^{-1}\{\breve{C}(e[u_1]\breve{I})\}+\omega]\}.
\end{aligned}
\end{equation}
We observe that
\begin{equation}\label{2.10}
\begin{aligned}
\breve{C}(e[u_1]\breve{I})&-\breve{C}(\lambda I+\breve{\eta}^{-1}\,\breve{C})^{-1}\,\breve{\eta}^{-1}\,\breve{C}(e[u_1]\breve{I})\\
=&\breve{C}(e[u_1]\breve{I})-\breve{\eta}^{-1}\,\breve{C}(\lambda I+\breve{\eta}^{-1}\,\breve{C})^{-1}\,\breve{C}(e[u_1]\breve{I})\\
=&[I-\breve{\eta}^{-1}\,\breve{C}(\lambda I+\breve{\eta}^{-1}\,\breve{C})^{-1}]\,\breve{C}\,(e[u_1]\breve{I})\\
=&\lambda(\lambda I+\breve{\eta}^{-1}\,\breve{C})^{-1}\,\breve{C}\,(e[u_1]\breve{I}).
\end{aligned}
\end{equation}
Combining \eqref{2.9} and \eqref{2.10}, we have
\begin{equation}\label{2.11}
\begin{aligned}
\lambda^2 u_1 =&\lambda \rho^{-1}{\rm div}\,\mathfrak{T}\{ (\lambda I+\breve{\eta}^{-1}\,\breve{C})^{-1}\,\breve{C}\,(e[u_1]\breve{I})\}\\
&-\rho^{-1}{\rm div}\,\mathfrak{T}\{\breve{C}(\lambda I+\breve{\eta}^{-1}\breve{C})^{-1}\omega\}+f_2+\lambda f_1\\
=&\lambda \rho^{-1}{\rm div}\,\mathfrak{T}\{ (\lambda I+\breve{\eta}^{-1}\,\breve{C})^{-1}\,\breve{C}\,(e[u_1]\breve{I})\}\\
&-\rho^{-1}{\rm div}\,\mathfrak{T}\{(\lambda I+\breve{\eta}^{-1}\breve{C})^{-1}\breve{C}\omega\}+f_2+\lambda f_1.
\end{aligned}
\end{equation}
The aforementioned Neumann boundary condition on $\Gamma_N$ applies with $(f_1,f_2,\omega)=0$. Thus, with \eqref{2.10},
\begin{equation}\label{2.12}
\begin{aligned}
\mathfrak{T}\{\breve{C}(e[u_1]\breve{I}-\phi)\}=&\lambda\mathfrak{T}\{ (\lambda I+\breve{\eta}^{-1}\,\breve{C})^{-1}\,\breve{C}\,(e[u_1]\breve{I})\}\\
=& \lambda\sum_{j=1}^n((\lambda I+\eta_j^{-1}\,C_j)^{-1}\,C_j)\,e[u_1]
\end{aligned}
\end{equation}
and the Neumann boundary condition becomes 
\begin{equation}\label{eq:NBDlambda}
\left\{\sum_{j=1}^n((\lambda I+\eta_j^{-1}\,C_j)^{-1}\,C_j)\,e[u_1]\right\}   
\nu=0\,\,\text{on $\Gamma_N$}.
\end{equation}

\section{ID system}

Here, we discuss the ID system associated with the EMM. Essentially, this system follows directly from (\ref{2.1}) upon setting $\phi^0 = 0$. With this initial condition, one can integrate
\begin{equation}\label{3.1}
\left\{
\begin{array}{ll}
\eta_j\partial_t\phi_j=\sigma_j[u,\phi_j]=C_j(e[u]-\phi_j),\\
\phi_j(0)=0,
\end{array}
\right.
\end{equation}
to yield
\begin{equation}\label{3.2}
\phi_j(t)=\int_0^t\, e^{-(t-s)\eta_j^{-1}\,C_j}\, \eta_j^{-1}\,C_j\, e[u(s)] ds,
\end{equation}
where the $x$-dependence of the different functions and tensors is suppressed. With \eqref{2.2}, we directly obtain
\begin{equation}\label{3.3}
\sigma[u]=\sum_{j=1}^n C_j \left\{
e[u] - \int_0^t\, e^{-(t-s)\eta_j^{-1}\,C_j}\, \eta_j^{-1}\,C_j\, e[u(s)] ds \right\},
\end{equation}
with $\phi$ being eliminated and signifying a description in terms of relaxation. The ID system for the EMM is then given by
\begin{equation}\label{VID sys}
\left\{
\begin{array}{ll}
\rho\,\partial_t^2 u - {\rm div} \, \sigma[u]=0,\\[0.25cm]
u=0\quad\text{on}\quad\Gamma_D,\\
\sigma[u]\nu=0\quad\text{on}\quad\Gamma_N,\\[0.25cm]
(u,\partial_t u)|_{t=0}=(u^0,v^0)\quad\text{on}\quad\Omega.
\end{array}
\right.
\end{equation}
As in Remark\ref{ESLS1}, we include the ESLSM by setting $\eta_j C_j = 0$ for some of the values of $j$.

Again, this system is equivalent to \eqref{2.1} subject to the restriction, $\phi^0 = 0$. This initial condition avoids the occurrence of stationary solutions. The mentioned reduced system, which will be discussed later, is a (closed) system given in terms of other dependent variables; the reduced system does not have stationary solutions either. However, it generates such solutions upon transforming the variables back to the original ones. We remark that this restriction has profound consequences such as preventing the generation of a semigroup. The well-posedness of this system in appropriate function spaces and the exponential decay of solutions in these function spaces are presented in a companion paper \cite{IDUD_2023}. In the above, the exponential is evaluated through an expansion and contractions of rank $4$ tensor $C_j$ with itself. 

Ahead of the further analysis of properties of solutions, we note that obtaining exponential decay is a nontrivial matter. In case that EMM includes a ESLS, neither the $\phi_j$'s nor the solution $u$ of the ID system exhibit this decay. In the case of a  pure EMM, the exponential decay of each $\phi_j$ depends on the relation between the lower bound of positive symmetric matrix $\eta_j^{-1} C_j$ and the exponential decay rate of solutions of the ID system. To be clear, one can view $\eta_j^{-1}C_j,\,1\le j\le n$ as a symmetric matrix employing the Voigt notation.

The Laplace transform of $\phi_j$ is given by
\begin{equation}\label{3.6}
\hat{\phi}_j(\lambda)=(\lambda I+\eta_j^{-1}\,C_j)^{-1}\, \eta_j^{-1}\,C_j\, e[\hat{u}(\lambda)],
\end{equation}
where $\hat{u}(\lambda)$ is the Laplace transform of $u(t)$. With this Laplace transform, we recover the Neumann boundary condition in \eqref{eq:NBDlambda}.
Indeed, combining \eqref{3.6} and \eqref{2.12} and noting that $\hat{u}_1 = \hat{u}$, we have
\begin{equation}\label{3.7}
\begin{aligned}
\widehat{\sigma[u,\phi]}(\lambda)&=\sum_{j=1}^nC_j(e[\hat{u}(\lambda)]-\hat{\phi}_j(\lambda))
=\mathfrak{T}\{\breve{C}(e[\hat{u}(\lambda)]\breve{I}-\hat{\phi}(\lambda))\}
\\
&=\sum_{j=1}^nC_je[\hat{u}(\lambda)]-\sum_{j=1}^n(\lambda I+\eta_j^{-1}\,C_j)^{-1}\, \eta_j^{-1}\,C_j^2\, e[\hat{u}(\lambda)]\\
&=\sum_{j=1}^n\{I-(\lambda I+\eta_j^{-1}\,C_j)^{-1}\, \eta_j^{-1}\,C_j\}C_je[\hat{u}(\lambda)]\\
&=\lambda\sum_{j=1}^n (\lambda I+\eta_j^{-1}\,C_j)^{-1}\, C_je[\hat{u}(\lambda)].
\end{aligned}
\end{equation}
Therefore, the Neumann boundary condition of the $\lambda$-equation for the ID system coincides with the one of the $\lambda$-equation for the AD system.

\section{Resolvent estimate}

In this section, we give a resolvent estimate for the AD system. First, using \eqref{2.6}, \eqref{2.8} and \eqref{2.11}, we rewrite the $\lambda-$equation \eqref{2.6} in an alternative form,
\begin{equation}\label{4.1}
\left\{
\begin{array}{rl}
\lambda^2 u_1&=\lambda \rho^{-1}{\rm div}\,\mathfrak{T}\{ (\lambda I+\breve{\eta}^{-1}\,\breve{C})^{-1}\,\breve{C}\,(e[u_1]\breve{I})\}\\
&\quad-\rho^{-1}{\rm div} \, \mathfrak{T}\{(\lambda I+\breve{\eta}^{-1}\breve{C})^{-1}\breve{C}\omega\}+f_2+\lambda f_1,\\
u_2&=\lambda u_1-f_1,\\
\phi&=(\lambda I+\breve{\eta}^{-1}\breve{C})^{-1}[\breve{\eta}^{-1}\{\breve{C}(e[u_1]\breve{I})\}+\omega].
\end{array}
\right.
\end{equation}
In the later analysis pertaining to $A$, we need to consider the resolvent equation for the $(u_1,u_2)$ components separately. In order to do that, we need to ensure sufficient regularity coming from the equation for $\phi$. More precisely, the term $-\rho^{-1}{\rm div} \, \mathfrak{T}\{(\lambda I + \breve{\eta}^{-1}\breve{C})^{-1}\breve{C}\omega\}$ in the first equation of \eqref{4.1} has to increase the regularity such that the outcome is in $L^2_\rho(\Omega)$. Hence, we have to smooth $\omega \in L^2(\Omega)$, most straightforwardedly with the isomorphism
$$ \zeta := (-\Delta)^{-1/2} :\ L^2(\Omega)\xrightarrow{\sim} H_0^1(\Omega),$$
where $\Delta$ is the Laplace operator on $\Omega$ supplemented with the Dirichlet boundary condition. We also define
\begin{equation}\label{4.2}
\mathfrak{Z}:=
\begin{pmatrix}
I_d                &        O      &        O\\
O                &        I_d      &        O\\
O                &        O      &        \mathfrak{z}
\end{pmatrix},
\end{equation}
with $\mathfrak{z}=\zeta \breve{I}$. Replacing $F=(f_1,f_2,\omega)^t\in L^2(\Omega)$ by $\mathfrak{Z}\,F=(f_1,f_2,\mathfrak{z} \omega)^t$, and then
composing $\mathfrak{Z}^{-1}$ with the $\lambda-$equation, $(\lambda I-A)U=\mathfrak{Z}F$, we obtain
\begin{equation}\label{4.3}
\begin{aligned}
\mathfrak{Z}^{-1}(\lambda I-A)U&=\lambda(\mathfrak{Z}^{-1} U)-\mathfrak{Z}^{-1} AU\\
&=\lambda V-\mathfrak{Z}^{-1} A \mathfrak{Z} V\\
&=(\lambda I- A_{Z})V=F,
\end{aligned}
\end{equation}
where $U=\mathfrak{Z} V$ and
\begin{equation}\label{4.4}
A_Z:=\mathfrak{Z}^{-1} A\mathfrak{Z}=
\begin{pmatrix}
O & I_d & O\\
\rho^{-1}{\rm div}\,[\mathfrak{T}\{\breve{C}(e[\cdot]\breve{I})\}]& O & -\rho^{-1}{\rm div}\,[\mathfrak{T}\{(\breve{C}\mathfrak{z})\cdot\}]\\
\mathfrak{z}^{-1}\breve{\eta}^{-1}\{\breve{C}(e[\cdot]\breve{I})\}& O & -\mathfrak{z}^{-1}\breve{\eta}^{-1}[(\breve{C}\mathfrak{z})\cdot]
\end{pmatrix}.
\end{equation}
We denote $V =(v_1,v_2,\phi_z)^t$, and let $F=(f_1,f_2,\omega)^t\in\mathcal{L}^2(\Omega) := L_\rho^2(\Omega)\times L^2(\Omega)\times L^2(\Omega)$ with $L_\rho^2(\Omega) := L^2(\Omega; \rho dx)$. Then $(\lambda I- A_{Z})V=F$ takes the componentwise form,
\begin{equation}\label{4.5}
\left\{
\begin{array}{ll}
\lambda v_1-v_2=f_1,\\
\lambda v_2-\rho^{-1}{\rm div}\,[\mathfrak{T}\{\breve{C}(e[v_1]\breve{I})\}]+\rho^{-1}{\rm div}\,[\mathfrak{T}\{(\breve{C}\mathfrak{z})\phi_z\}]=f_2\\
\lambda \phi_z-\mathfrak{z}^{-1}\breve{\eta}^{-1}\{\breve{C}(e[v_1]\breve{I})\}+\mathfrak{z}^{-1}\breve{\eta}^{-1}[(\breve{C}\mathfrak{z})\phi_z]=\omega.
\end{array}
\right.
\end{equation}
Substituting $(v_1,v_2,\phi_z) = (u_1,u_2,\mathfrak{z}^{-1}\phi)$, a computation yields
\begin{equation}\label{4.6}
\left\{
\begin{aligned}
\lambda u_1=& u_2+f_1,\\
\lambda u_2 =& \lambda\rho^{-1}{\rm div}\,\mathfrak{T}\{ (\lambda I+\breve{\eta}^{-1}\,\breve{C})^{-1}\,\breve{C}\,(e[u_1]\breve{I})\}\\
&-\rho^{-1}{\rm div}\,\mathfrak{T}\{(\lambda I+\breve{\eta}^{-1}\breve{C})^{-1}\breve{C}(\mathfrak{z}\omega)\}+f_2,\\
\lambda \phi=& \lambda(\lambda I+\breve{\eta}^{-1}\breve{C})^{-1}[\breve{\eta}^{-1}\{\breve{C}(e[u_1]\breve{I})\}+\mathfrak{z}\omega],
\end{aligned}
\right.
\end{equation}
which can be written in the matrix form
\begin{equation}\label{eq:tAeq}
(\lambda I-\tilde{A})U=
(f_1,f_2-\rho^{-1}{\rm div}\,\mathfrak{T}\{(\lambda I+\breve{\eta}^{-1}\breve{C})^{-1}\breve{C}(\mathfrak{z}\omega)\},\lambda(\lambda I+\breve{\eta}^{-1}\breve{C})^{-1}(\mathfrak{z}\omega))^t,
\end{equation}
with
\begin{equation}\label{A circle}
\tilde{A}:=
\begin{pmatrix}
O  & I_d  & O\\
\lambda\rho^{-1}{\rm div}\,\mathfrak{T}\{(\lambda I+\breve{\eta}^{-1}\breve{C})^{-1}(\breve{C}e[\cdot]\breve{I})\}& O & O\\
\lambda(\lambda I+\breve{\eta}^{-1}\breve{C})^{-1}\breve{\eta}^{-1}\{\breve{C}(e[\cdot]\breve{I})\}& O & O
\end{pmatrix}.
\end{equation}
We find the equivalence of $\lambda$-equations
\[
\text{\eqref{4.3}}\ \Leftrightarrow\
\text{\eqref{eq:tAeq}}.
\]
We note that the right-most column of $\tilde{A}$ in the blockwise sense is zero. We will later appreciate that this form of $\tilde{A}$ implies that $0$ is not in the resolvent set of $A_Z$.

The remainder of this section is devoted to proving that $A_Z$ generates a $C_0$-semigroup. This is accomplished by proving the following estimate upon analyzing the relevant $\lambda$-equation. 

\begin{pr} There exists a constant $\beta>0$ such that 
\begin{equation}\label{Hille-Yosida type resolvent estimate}
\Vert(\lambda I-A_Z)^{-1}\Vert\le (|\lambda|-\beta)^{-1},\quad |\lambda|>\beta,    
\end{equation}
where $\Vert\,\cdot\Vert$ is the operator norm on the space $W_Z$ introduced in \eqref{Wz} below.
\end{pr}

\begin{proof}
We begin with introducing some notation. Let $W_\lambda:= S_\lambda\times H^1_0(\Omega)$ with $S_\lambda:=K(\Omega)\times L_\rho^2(\Omega)$, where $K(\Omega):=\{u\in H^1(\Omega): u|_{\Gamma_D}=0\}$. $S_\lambda$ is a Hilbert space equipped with the inner product $(\cdot,\cdot)_{S_\lambda}$ dependent on $\lambda$ and given as
\begin{equation}\label{4.8}
\begin{aligned}
(V,V')_{S_\lambda}:=(B_\lambda\nabla v_1, \nabla v_1')+(v_2,v_2')_\rho, \quad V=(v_1,v_2),\quad V'=(v_1',v_2'),
\end{aligned}
\end{equation}
where
\[
   B_\lambda\nabla v_1:=\lambda \mathfrak{T}\{(\lambda I+\breve{\eta}^{-1}\,\breve{C})^{-1}(\breve{C}e[v_1]\breve{I})\},
\]
and $(\cdot,\cdot)_{\rho}$ and $(\cdot,\cdot)$ are the inner products on $L_{\rho}^2(\Omega)$ and $L^2(\Omega)$, respectively.
In the further analysis, we will consider the asymptotic behaviors, $B_\infty$ and $S_\infty$, of $B_\lambda$ and $S_\lambda$, which are defined via $B_\infty = \lim_{\lambda\rightarrow \infty}B_\lambda$ and the induced inner product. We have
\begin{equation}\label{4.9}
   B_\infty\nabla v_1 = \mathfrak{T}\{(\breve{C}e[v_1]\breve{I})\} = C \nabla v_1,\quad
   C = \sum_{j=1}^n C_j.
\end{equation}
We let $W_{\infty} = S_\infty \times H_0^1(\Omega)$ but will also write $W$ for $W_{\infty}$.

We consider the upperleft submatrix, $\tilde{A}_u$, of $\tilde{A}$, which has the form
\begin{equation}\label{4.10}
\tilde{A}_u:=
\begin{pmatrix}
O & I_d \\
\lambda\rho^{-1}{\rm div}\,\mathfrak{T}\{(\lambda I+\breve{\eta}^{-1}\breve{C})^{-1}(\breve{C}e[\cdot]\breve{I})\}& O
\end{pmatrix}
= \begin{pmatrix}
O & I_d \\
\rho^{-1} {\rm div}(B_\lambda\nabla[\cdot]) & O
\end{pmatrix}.
\end{equation}
We consider the equation $\tilde{A}_u Y_u = G$ with $G = (\hat{f_1},\hat{f_2})^t$ and $Y_u = (y_1,y_2)^t$, which is equivalent to $y_2 = \hat{f_1}$, $\rho^{-1} {\rm div}(B_\lambda\nabla y_1) = \hat{f_2}$. The second equation is related to the Neumann boundary condition on $\Gamma_N$. In terms of the continuous sesquilinear form
\begin{equation}\label{sesquilinear form a'}
\mathfrak{a}'(y_1,z):=\int_\Omega(B_\lambda\nabla y_1)\overline{\nabla z},\,\,z\in K(\Omega),
\end{equation}
we consider the linear map $K(\Omega)\ni y_1\rightarrow \rho^{-1}{\rm div}(B_\lambda\nabla y_1)$
as the bounded linear operator $\mathfrak{A}': K(\Omega)\rightarrow K(\Omega)'$ defined by
$$
\mathfrak{a}'(y_1,z)=(\mathfrak{A}'y_1,z)_\rho,\,\,z\in K(\Omega),
$$
where $K(\Omega)'$ is the dual space of $K(\Omega)$ with respect to the inner product $(\cdot\,,\cdot)_\rho$.
Furthermore, we let $\mathfrak{A} :\ S_\lambda \to S_\lambda'$ be the operator defined by replacing the left lower block of the operator matrix in \eqref{4.10} by $\mathfrak{A}'$, that is,
\[
\mathfrak{A} = \begin{pmatrix}
    O & I_d \\ \mathfrak{A}' & O
\end{pmatrix},
\]
where $S_\lambda'$ is the dual space of $S_\lambda$ with respect to the $L^2(\Omega) \times L^2_\rho(\Omega)$ inner product. Then we identify $\tilde{A}_u$ with $\mathfrak{A}|_{D(\tilde{A}_u)}$,
with
\[
 D(\tilde{A}_u):=\{Y_u=(y_1,y_2)^{\it t}\in S_\lambda : y_2\in H^1(\Omega),\,\mathfrak{A}{Y_u}\in L^2(\Omega) \times L^2_\rho(\Omega)\}.
\]
By the coercivity of the sesquilinear form \eqref{sesquilinear form a'} due to Korn's inequality \cite{Devaut}, we have $\overline{D(\tilde{A}_u)} = S_\lambda$ \cite[Chapter 3, Section 2]{S}.

An elementary calculation leads to
\begin{multline}\label{4.15}
(\tilde{A}_u{Y_u},{Y_u})_{S_\lambda}
=(B_\lambda\nabla y_2, \nabla y_1)
+({\rm div}\{B_\lambda\nabla y_1\},y_2)\\
=-(y_2, {\rm div}\{B_\lambda\nabla y_1\})+({\rm div}\{B_\lambda\nabla y_1\},y_2)\\
= (B_\lambda\nabla y_1, \nabla y_2)+(y_2,{\rm div}\{B_\lambda\nabla y_1\})
= ({Y_u},\tilde{A}_u{Y_u})_{S_\lambda}
\end{multline}
for ${Y_u}\in  D(\tilde{A}_u)$. Using this equality, we obtain for $\lambda>0$ and ${Y_u} \in  D(\tilde{A}_u)$ the estimate,
\begin{equation}\label{4.16}
\begin{aligned}
\|(\lambda I-\tilde{A}_u){Y_u}\|_{S_\lambda}^2&=((\lambda I-\tilde{A}_u){Y_u},(\lambda I-\tilde{A}_u){Y_u})_{S_\lambda}\\
&=\lambda^2\|{Y_u}\|_{S_\lambda}^2+\|\tilde{A}_u{Y_u}\|_{S_\lambda}^2\\
&\geq \lambda^2\|{Y_u}\|_{S_\lambda}^2.
\end{aligned}
\end{equation}

\medskip

\noindent
\textbf{The case $\lambda>0$}. For any $\lambda>0$, the bijectivity of the map
$$ \lambda I-\tilde{A}_u : D(\tilde{A}_u) \rightarrow S_\lambda$$
can be shown using the unique solvability of the aforementioned variational problem by recalling the Korn inequality. Hence, in terms of the operator norm, \eqref{4.16} implies
\begin{equation}\label{4.17}
\\|(\lambda I-\tilde{A}_u)^{-1}\|
\leq \lambda^{-1}, \quad \lambda>0.
\end{equation}

Now, we lift $\tilde{A}_u$ to $W_\lambda:=S_\lambda\times H^1_0(\Omega)$ to define an operator which is nothing but $\tilde{A}$, which clearly has $\overline{D(\tilde{A})}=W_\lambda$, where $D(\tilde{A})$ is the domain of $\tilde A$ simply given as
$$ 
D(\tilde A):=D(\tilde{A}_u)\times H_0^1(\Omega).
$$
For any $\lambda>0$,
$$ \lambda I-{\tilde{A}} : D(\tilde A) \to W_\lambda $$
is bijective.

Combining the relevant components of \eqref{eq:tAeq} and \eqref{4.17}, we find that
\begin{equation}\label{4.18}
\begin{aligned}
\|(u_1,u_2)\|_{S_\lambda}&\leq \lambda^{-1}\|(f_1,f_2-\rho^{-1}{\rm div}\,\mathfrak{T}\{(\lambda I+\breve{\eta}^{-1}\breve{C})^{-1}\breve{C}(\mathfrak{z}\omega)\})\|_{S_\lambda}\\
&\leq \lambda^{-1}\|(f_1,f_2)\|_{S_\lambda}+m_1\lambda^{-2}\|\mathfrak{z}\omega\|_{H^1_0(\Omega)}
\end{aligned}
\end{equation}
for some positive constant, $m_1$. With the third equation of \eqref{4.6} and \eqref{4.18}, we get
\begin{equation}\label{4.19}
\begin{aligned}
\|\phi\|_{H^1_0(\Omega)} &=\|(\lambda I+\breve{\eta}^{-1}\breve{C})^{-1}[\breve{\eta}^{-1}\{\breve{C}(e[u_1]\breve{I})\}+\mathfrak{z}\omega]\|_{H^1_0(\Omega)}\\
&\leq m_2\lambda^{-1}\|(u_1,u_2)\|_{S_\lambda}+\lambda^{-1}(1+m_3\lambda^{-1})\|\mathfrak{z}\omega\|_{H^1_0(\Omega)}\\
&\leq m_4 \lambda^{-2}\|(f_1,f_2)\|_{S_\lambda}+\lambda^{-1}(1+m_5\lambda^{-1})\|\mathfrak{z}\omega\|_{H^1_0(\Omega)},
\end{aligned}
\end{equation}
where $m_2, m_3, m_4, m_5$ are positive constants chosen appropriately.

We introduce $m$ according to
\begin{equation}\label{def m}
 m:=m_1+m_4+m_5.   
\end{equation}
For $\lambda > m$, we have the elementary estimate
\begin{equation}\label{4.20}
\begin{aligned}
\lambda^{-1}+m\lambda^{-2}&= (\lambda-m)^{-1}(\lambda-m)\lambda^{-1}(1+m\lambda^{-1})\\
&=(\lambda-m)^{-1}(1-m\lambda^{-1})(1+m\lambda^{-1})\\
&=(\lambda-m)^{-1}(1-m^2\lambda^{-2})\\
&\leq (\lambda-m)^{-1}.
\end{aligned}
\end{equation}
Combining \eqref{4.18}, \eqref{4.19} and \eqref{4.20}, we then obtain
\begin{equation}\label{4.21}
\begin{aligned}
\|U\|_{W_\lambda}&\leq \lambda^{-1}(1+m_4\lambda^{-1})\|(f_1,f_2)\|_{S_\lambda}+\lambda^{-1}(1+m_1\lambda^{-1}+m_5\lambda^{-1})\|\mathfrak{z}\omega\|_{H^1_0(\Omega)}\\
&\leq (\lambda^{-1}+m\lambda^{-2})\|(f_1,f_2)\|_{S_\lambda}+(\lambda^{-1}+m\lambda^{-2})\|\mathfrak{z}\omega\|_{H^1_0(\Omega)}\\
&\leq (\lambda^{-1}+m\lambda^{-2})\|\mathfrak{Z}F\|_{W_\lambda}\\
&\leq (\lambda-m)^{-1}\|\mathfrak{Z}F\|_{W_\lambda}.
\end{aligned}
\end{equation}
However, the norm of $W_\lambda$ depends on $\lambda$. To proceed, we need an estimate similar to \eqref{4.21} with respect a norm independent on $\lambda$. To this end, we consider the asymptotic behavior of the norm of $W_\lambda$ as $|\lambda| \to \infty$.

We revisit \eqref{4.8} and note that the asymptotic behavior of $B_{\lambda}$ determines how the desired estimate can be obtained. We have
\begin{equation}\label{4.22}
\begin{aligned}
(B_\lambda\nabla v_1,\nabla v_1)&=(\lambda \mathfrak{T}\{(\lambda I+\breve{\eta}^{-1}\,\breve{C})^{-1}(\breve{C}e[v_1]\breve{I})\},\nabla v_1)\\
&= (\mathfrak{T}\{( I+\lambda^{-1}\breve{\eta}^{-1}\,\breve{C})^{-1}(\breve{C}e[v_1]\breve{I})\},\nabla v_1)\\
&\leq (\mathfrak{T}\{\breve{C}e[v_1]\breve{I}\},\nabla v_1)+l_1\lambda^{-1}(\mathfrak{T}\{\breve{C}e[v_1]\breve{I}\},\nabla v_1)\\
&\leq (C\nabla v_1,\nabla v_1)+l_2\lambda^{-1}\|\nabla y_1\|^2_{L^2(\Omega)}
\end{aligned}
\end{equation}
with positive constants $l_1,l_2$. Since 
\begin{equation*}
\begin{aligned}
\|Y\|^2_{W}=(C\nabla y_1,\nabla y_1)+\|y_2\|^2_\rho+\|y_3\|^2_{H^1_0(\Omega)},
\end{aligned}
\end{equation*}
and using \eqref{4.22}, we obtain the estimates
\begin{equation}\label{4.23}
\begin{aligned}
\|Y\|^2_{W}&\leq (1+k_1\lambda^{-1})\|Y\|^2_{W_\lambda},\\
\|Y\|^2_{W_\lambda}&\leq (1+k_2\lambda^{-1})\|Y\|^2_{W},
\end{aligned}
\end{equation}
where $k_1>0, k_2>0$ are some constants.
Combining \eqref{4.21} and \eqref{4.23}, we have, for $\lambda > m\geq 1$, that
\begin{equation}\label{4.24}
\begin{aligned}
\|U\|^2_{W}&\leq (1+k_1\lambda^{-1})\|U\|^2_{W_\lambda}\\
&\leq (1+k_1\lambda^{-1})(\lambda-m)^{-2}\|\mathfrak{Z}F\|^2_{W_\lambda}\\
&\leq (1+k_1\lambda^{-1})(\lambda-m)^{-2}(1+k_2\lambda^{-1})\|\mathfrak{Z}F\|^2_{W}\\
&\leq (1+k\lambda^{-1})(\lambda-m)^{-2}\|\mathfrak{Z}F\|^2_{W},
\end{aligned}
\end{equation}
where $k$ is given by
\begin{equation}\label{def k}
k:=1+k_1+k_2+k_1k_2.
\end{equation}

For $\lambda > m+k$, we have the elementary estimate,
\begin{equation}\label{4.25}
\begin{aligned}
(1+k\lambda^{-1})(\lambda-m)^{-2}&\leq (1+k(\lambda-m)^{-1})(\lambda-m)^{-2}\\
&=  (\lambda-m-k)^{-2} (\lambda-m-k)^{2}(\lambda-m)^{-2}(1+k(\lambda-m)^{-1})\\
&=  (\lambda-m-k)^{-2} (1-k(\lambda-m)^{-1})^{2}(1+k(\lambda-m)^{-1})\\
&\leq  (\lambda-m-k)^{-2} (1-k(\lambda-m)^{-1})(1+k(\lambda-m)^{-1})\\
&\leq  (\lambda-m-k)^{-2}.
\end{aligned}
\end{equation}
Combining \eqref{4.24} and \eqref{4.25}, we have that for $\lambda > m+k$, the following estimate holds,
\begin{equation}\label{4.26}
\begin{aligned}
\|U\|^2_{W}&\leq (\lambda-m-k)^{-2}\|\mathfrak{Z}F\|^2_{W}.
\end{aligned}
\end{equation}
In terms of $V$, \eqref{4.26} is given as
\begin{equation}\label{4.27}
\Vert V\Vert_{W_Z}\le (\lambda-m-k)^{-1}\Vert F\Vert_{W_Z},\,\,\lambda>m+k,
\end{equation}
where $W_Z$ and the domain $D(A_Z)$ are defined as 
\begin{equation}\label{Wz}
W_Z:=S_\infty\times L^2(\Omega),\,\,
D(A_Z):=\{V^0\in W_Z: A_z V^0\in\mathcal{L}^2(\Omega)\}.
\end{equation}

\medskip

\noindent
\textbf{The case $\lambda<0$}. It is clear that if $\lambda\leq -l_0$ for a large enough $l_0$, then $ (V,V')_{S_\lambda} $ in \eqref{4.8} is a norm. We repeat the arguments from \eqref{4.8} to \eqref{4.16}, and conclude that 
\begin{equation}\label{4.31}
\begin{aligned}
\|(\lambda I-\tilde{A}_u)^{-1}\|, \quad \lambda\leq -l_0.
\end{aligned}
\end{equation}
Combining \eqref{eq:tAeq} and \eqref{4.31}, we have for $\lambda\leq -l_1$, $l_1$ sufficiently large, 
\begin{equation}\label{4.32}
\begin{aligned}
\|(u_1,u_2)\|_{S_\lambda}&\leq (-\lambda)^{-1}\|(f_1,f_2-\rho^{-1}{\rm div}\,\mathfrak{T}\{(\lambda I+\breve{\eta}^{-1}\breve{C})^{-1}\breve{C}(\mathfrak{z}\omega)\})\|_{S_\lambda}\\
&\leq (-\lambda)^{-1}\|(f_1,f_2)\|_{S_\lambda}+m_1\lambda^{-2}\|\mathfrak{z}\omega\|_{H^1_0(\Omega)};
\end{aligned}
\end{equation}
here, $m_1 \geq l_1 > 0$ in \eqref{4.18} large so that we can use  a common notation. With the third equation of \eqref{4.6} and \eqref{4.32}, we obtain, for $\lambda\leq -l_1$,
\begin{equation}\label{4.33}
\begin{aligned}
\|\phi\|_{H^1_0(\Omega)} &=\|(\lambda I+\breve{\eta}^{-1}\breve{C})^{-1}[\breve{\eta}^{-1}\{\breve{C}(e[u_1]\breve{I})\}+\mathfrak{z}\omega]\|_{H^1_0(\Omega)}\\
&\leq m_2(-\lambda)^{-1}\|(u_1,u_2)\|_{S_\lambda}+(-\lambda)^{-1}(1+m_3(-\lambda)^{-1})\|\mathfrak{z}\omega\|_{H^1_0(\Omega)}\\
&\leq m_4 \lambda^{-2}\|(f_1,f_2)\|_{S_\lambda}+(-\lambda)^{-1}(1+m_5(-\lambda)^{-1})\|\mathfrak{z}\omega\|_{H^1_0(\Omega)}.
\end{aligned}
\end{equation}
We have taken $m_2, m_3, m_4, m_5$ in \eqref{4.19} large to be able using a common notation. We then repeat the arguments from \eqref{4.20} to \eqref{4.27}, and get
\begin{equation}\label{4.34}
\Vert V\Vert_{W_Z}\le (-\lambda-m-k)^{-1}\Vert F\Vert_{W_Z},\,\,-\lambda>m+k.
\end{equation}

\medskip\medskip

\noindent
Combining \eqref{4.27} and \eqref{4.34}, we find that
\begin{equation}\label{4.35}
\Vert V\Vert_{W_Z}\le (|\lambda|-m-k)^{-1}\Vert F\Vert_{W_Z},\,\,|\lambda|>m+k.
\end{equation}
Therefore, we obtained \eqref{Hille-Yosida type resolvent estimate} for $A_Z$ with $\beta:=m+k$.
\end{proof}

\section{$C_0$-group for the AD System}

In this section, based on the resolvent estimates in the previous section, we will discuss the generation of a $C_0$-group for the AD system. For that, we first recall the following standard theorem (see Corollary of Theorem 5.6 on page 296 \cite{Mizo} or Theorem 6.3 on page 23 \cite{Pazy}).

\begin{theorem}\label{cor5.2}
Let $\mathcal{A}$ be a closed operator on a Banach space $X$ having a dense domain of definition in $X$. If there exists $\beta\ge 0$ such that for
 $|\lambda|>\beta$, the resolvent $(\lambda I-\mathcal{A})^{-1}$ of $\mathcal{A}$ exists and satisfies
\begin{equation}\label{5.3}
\begin{aligned}
\|(\lambda I-\mathcal{A})^{-1}\|\leq (|\lambda|-\beta)^{-1},\,\,\,|\lambda|>\beta,
\end{aligned}
\end{equation}
then $\mathcal{A}$ generates a $C_0$-group on $X$.
\end{theorem}

With Theorem \ref{cor5.2}, the resolvent estimate in the previous section implies the following.

\begin{theorem}\label{semigroup Az}
$A_Z$ generates a $C_0$-group on $W_Z$ and $\rho(A_Z)\supset\{\lambda\in{\mathbb R}: |\lambda|>\beta\}$, where $\rho(A_Z)$ denotes the resolvent set of $A_Z$, $\beta=m+k$ and $m$ and $k$ are defined in \eqref{def m} and \eqref{def k}, respectively.   
\end{theorem}

\paragraph{Abstract Cauchy problem} As an immediate consequence of this theorem, we obtain

\begin{theorem}\label{inhomog CP}${}$
\begin{itemize}
\item [{\rm (i)}] 
Let $F\in C^1([0,\infty);W_Z)$ and consider the following Cauchy problem
\begin{equation}\label{abstract C-problem}
\left\{
\begin{array}{ll}
\displaystyle\frac{d}{dt}\,V(t)=A_Z V(t)+F(t),\quad t>0,\\
\\
V(0)=V^0\in D(A_Z).
\end{array}
\right.
\end{equation}
Then, there exists a unique strong solution $V=V(t)\in C^0([0,\infty);W_Z)\cap C^1((0,\infty);W_Z)$ of \eqref{abstract C-problem}. Here, in addition to the usual conditions for the solution of \eqref{abstract C-problem}, the strong solution $V(t)$ has to be differentiable almost everywhere in $(0,\infty)$ and $dV(t)/dt\in L^1((0,T); W_Z)$ for each $T>0$. 
\item [{\rm (ii)}]
Concerning the regularity of the solution of \eqref{abstract C-problem}, let $F\in C^{m+1}([0,\infty);W_Z)$ for $m\in{\mathbb N}$ and assume that the condition $V^\ell\in D(A_Z),\,1\le l\le m$, referred to as the compatibility condition of order $m$, holds. Here, the $V^\ell$'s are defined as
$$
V^\ell:=A_Z V^{l-1}+F^{(l-1)}(0)\,\,\text{\rm with $F^{(l-1)}:=\frac{d^{l-1}F}{dt^{l-1}}$}.
$$
Consider  
\begin{equation}\label{formula for V}
V(t):=\sum_{l=0}^{m-1}\frac{t^l}{l!}V^l+\int_0^t\frac{(t-s)^{m-1}}{(m-1)!} \tilde V(s)\, ds,
\end{equation}
where $\tilde V\in C^0([0,\infty);W_Z)\cap C^1((0,\infty);W_Z)$ is the unique strong solution to the Cauchy problem
\begin{equation}\label{abstract C-problem for tilde V}
\left\{
\begin{array}{ll}
\displaystyle\frac{d}{dt}\,\tilde V(t)=A_Z \tilde V(t)+F^{(m)}(t),\quad t>0,\\
\\
\tilde V(0)=V^m.
\end{array}
\right.
\end{equation}
Then, $V=V(t)\in C^{m}([0,\infty); W_Z)\cap C^{m+1}((0,\infty);W_Z)$ is the unique strong solution to \eqref{abstract C-problem} in the space
$C^0([0,\infty); W_Z)\cap C^1((0,\infty); W_Z)$.
\item[{\rm (iii)}] Statements {\rm (i)} and {\rm (ii)} hold upon replacing $t>0$, $[0,\infty)$ and $(0,\infty)$ by $t<0$, $(-\infty,0]$ and $(-\infty,0)$, respectively. Furthermore, the two solutions in the above two different intervals match at $t=0$ up to $m$-th order derivatives if $F(t)\in C^{m+1}({\mathbb R}; W_Z)$, and this $F(t)$ together with $V^0\in D(A_Z)$ satisfy the compatibility condition of order m. Hence, {\rm (ii)} implies the existence of a unique strong solution $V\in C^m({\mathbb R}; W_z)$ to \eqref{abstract C-problem}.
\end{itemize}
\end{theorem}

We refer to Theorem 5.6 \cite{Mizo} and Section 4.2 \cite{Pazy} for the study of the abstract Cauchy problem and the generation of a $C_0$-semigroup; we mention Section 1.6 \cite{Pazy}, where the relation between the $C_0$-semigroup and the $C_0$-group is established. As for \eqref{formula for V}, there is a similar formula in \cite{Dafermos, Ikawa}. By quite a formal argument, except for verifying the commutativity of $A_Z$ and $\int_0^t\ \cdot\ ds$, it follows that $V(t)$ given by \eqref{formula for V} is the unique solution to \eqref{abstract C-problem} in the space
$C^0([0,\infty); W_Z)\cap C^1((0,\infty); W_Z)$. The verification of the mentioned commutativity can be shown by using that $A_Z$ is a closed operator.

We conclude this section by observing that $0\not\in\rho(A_Z)$, that is, $0\not\in\rho(A)$. This yields an obstruction to guaranteeing decay properties of the solutions, which would require that  $\{i \,\mu\ :\ \mu \in {\mathbb R}\} \subset \rho(A_Z)$. To show this, let $U\in W$ satisfy $AU=0$. Then, using \eqref{2.6}, we have
\begin{equation}
\left\{
\begin{array}{rl}
u_2&=0,\\
{\rm div}\,\mathfrak{T}\{\breve{C}\,(e[u_1]\breve{I})\}-{\rm div}\,\mathfrak{T}\{\breve{C}\phi\}&=0,\\
\breve{C}(e[u_1]\breve{I})-\breve{C}\phi&=0.
\end{array}
\right.
\end{equation}
Now, let $0\not=\phi\in L^2(\Omega)$ and search for a $u_1$ that satisfies $e[u_1]\breve{I}=\phi$. Then $u_1$ must be given as the unique solution $u_1\in K(\Omega)$ of the following boundary value problem:
\begin{equation}\label{stationary sol}
 \left\{
\begin{array}{rl}
\text{\rm div}\,\mathfrak{T}\{\breve{C}(e[u_1]\breve{I}\}&=\text{\rm div}\,\mathfrak{T}\{\breve{C}\phi\},\\[0.25cm]
u_1&=0\quad\text{on $\Gamma_D$},\\
(\mathfrak{T}\{\breve{C}(e[u_1]\breve{I})\})\nu&=
(\mathfrak{T}\{\breve{C}\phi\})\nu\quad\text{\rm on $\Gamma_N$}.
\end{array}
 \right.
\end{equation}
We note that the above boundary condition on $\Gamma_N$ comes from the boundary condition $\sigma[u_1,\phi]\nu=0$ on $\Gamma_N$, which is consistent with the occurrence of the inhomogeneous term in the first equation of \eqref{stationary sol}. As a consequence, we have $0\not\in\rho(A)$.


\section{The reduced system, $C_0$-group and contractive $C_0$-semigroup}

To mitigate the fact that $0\not\in\rho(A)$, we introduce a reduction of the original system. We then establish that this system generates a $C_0$-group.

\subsection*{The reduced system}

We let $v=\partial_t u$ and $\psi=e[u]\breve{I}-\phi$. We observe that \eqref{2.1} contains the following closed subsystem,
\begin{equation}\label{6.1}
\left\{
\begin{array}{ll}
\partial_tv=\rho^{-1}{\rm div}\mathfrak{T}\{\breve{C}\psi\},\\
\partial_t\psi=-\breve{\eta}^{-1}\breve{C}\psi+e[v]\breve{I},\\[0.25cm]
v=0\quad\text{on}\quad\Gamma_D, \quad
(\mathfrak{T}\{\breve{C}\psi\})\nu=0\quad\text{on}\quad \Gamma_N,\\[0.25cm]
(v,\psi)=(v^0,e[u^0]\breve{I}-\phi^0)\quad\text{at $t=0$}.
\end{array}
\right.
\end{equation}

\begin{remark}
By assuming that the initial value for $u$ in \eqref{6.1} is $u^0$, we obtain a solution $(u,v,\phi)$ of \eqref{2.1} from the relation $v=\partial_t u,\,\psi=e[u]\breve{I}-\phi$.
\end{remark}

\medskip

Next, we rewrite this initial boundary value problem as an abstract Cauchy problem. To begin with, we let
\begin{equation}\label{6.2}
L\cdot=
\begin{pmatrix}
0                &        \rho^{-1}{\rm div}\,\mathfrak{T}\{\breve{C}\cdot\} \\
e[\cdot]\breve{I}    &        -\breve{\eta}^{-1}\breve{C}\cdot
\end{pmatrix}
\end{equation}
and its domain $D(L)$ be given as
$$ D(L):=\{ (v,\psi)\in K(\Omega)\times L^2(\Omega) : L
\begin{pmatrix}
v                \\
\psi
\end{pmatrix}
\in L_\rho^2(\Omega)\times L^2(\Omega)\}. $$
We equip Hilbert space $H:=L^2_\rho(\Omega)\times L^2(\Omega)$ with the inner product
\begin{equation}\label{H-inner product}
(V,V')_H:=(v,v')_\rho+(\psi,\psi')_{\breve{C}}
\end{equation}
for $V=(v,\psi)^{\it t},\,\,V'=(v',\psi')^{\it t}$, where  $(\psi,\psi')_{\breve{C}}:=(\breve{C}\psi,\psi')$ for $\psi,\,\psi'\in L^2(\Omega)$. This inner product is equivalent to the standard inner product of $L_\rho^2(\Omega)\times L^2(\Omega)$. It can be shown that $\overline{D(L)}=H$ in a way similar to how this was done for $\tilde{A}_u$. Then, the abstract Cauchy problem takes the form
\begin{equation}\label{abstract CP}
\left\{
\begin{array}{ll}
\partial_t \begin{pmatrix}
v                \\
\psi
\end{pmatrix} =L \begin{pmatrix}
v                \\
\psi
\end{pmatrix},\\
\\
(v,\psi) =(v^0,\psi^0)\quad\text{\rm at $t=0$},
\end{array}
\right.
\end{equation}
where $\psi^0:=e[u^0]\breve{I}-\phi^0$.

We consider the $\lambda-$equation associated with the first equation of \eqref{abstract CP}:
\begin{equation}\label{6.4}
(\lambda I-L)
\begin{pmatrix}
v                \\
\psi
\end{pmatrix}
=
\begin{pmatrix}
f                \\
\omega
\end{pmatrix}\in H\,\,
\text{with $\lambda=\sigma+i\mu,\,\,\sigma, \mu\in{\mathbb R}$}.
\end{equation}
This equation is equivalent to the system
\begin{equation}\label{6.5}
\left\{
\begin{aligned}
\lambda v&=\rho^{-1}{\rm div}\mathfrak{T}\{\breve{C}\psi\}+f,\\
\lambda\psi&= -\breve{\eta}^{-1}\breve{C}\psi+e[v]\breve{I}+\omega,
\end{aligned}
\right.
\end{equation}
supplemented with the boundary condition given in \eqref{6.1}. We recall the definition of the reduced system. Since $\eta_j^{-1} C_j>0$ for $1\le j\le n$, there exists a $\delta>0$ such that 
\begin{equation}\label{cond-sigma}
\sigma I+\eta_j^{-1}C_j\ge\delta,\,\,1\le j\le n\,\, 
\end{equation}
holds for $\sigma\geq-\delta_0$ with $\delta_0>0$. Given $v$, $\psi$ can be obtained from the second equation of \eqref{6.5} based on \eqref{cond-sigma}. Substituting the result into the first equation of \eqref{6.5}, gives
\begin{equation}\label{6.8bis}
 \lambda v=\rho^{-1}\text{\rm div}\mathfrak{T}\{\breve{C}(\lambda\breve{I}+\breve{\eta}^{-1}\breve{C})^{-1}(e[v]\breve{I})\}+\rho^{-1}\text{\rm div}\mathfrak{T}\{\breve{C}(\lambda\breve{I}+\breve{\eta}^{-1}\breve{C})^{-1}\omega\}+f. 
\end{equation}
The boundary condition, $(\breve{C}\psi)\nu=0$ on $\Gamma_N$, becomes
\begin{equation}\label{new boundary condition}
\{\mathfrak{T}(\{\breve{C}(\lambda\breve{I}+\breve{\eta}^{-1}\breve{C})^{-1}(e[v]\breve{I})\})\}\nu=-\mathfrak{T}\{\breve{C}(\lambda\breve{I}+\breve{\eta}^{-1}\breve{C})^{-1}\omega\}\nu\quad\text{on $\Gamma_N$}.
\end{equation}
It follows readily that for a positive symmetric $K$,
\begin{equation}\label{split new1}
\begin{array}{ll}
(i\mu I+K)^{-1}=Q+iR,\\
\\
Q=K^{-1}-\mu^2 K^{-1}(\mu^2 I+K^2)^{-1},\,\,R=-\mu(\mu^2I+K^2)^{-1}.
\end{array}
\end{equation}
Applying this identity to $(\lambda\breve{I}+\breve{\eta}^{-1}\breve{C})^{-1}=(i\mu\breve{I}+K)^{-1}$,
leads to
\begin{equation}\label{split new2}
 \lambda v=\rho^{-1}\text{\rm div}\mathfrak{T}\{\mathcal{I}(e[v]\breve{I})\}+\rho^{-1}\text{\rm div}\mathfrak{T}\{\mathcal{I}\omega\}+f,  
\end{equation}
where $\mathcal{I}$ is given by
\begin{equation}\label{new I}
\mathcal{I}:=(K-i\mu\breve{I})(\mu^2\breve{I}+K^2)^{-1}\breve{C}.
\end{equation} 

\subsection*{Variational form}

We introduce a bilinear form, $B(v,w)$, on $K(\Omega)$ as follows,
\begin{equation}\label{bilinear-B}
B(v,w):=(\mathfrak{T}\{\mathcal{I} (e[v]\breve{I})\},e[w])+\lambda(v,w)_\rho,\,\,v,w\in K(\Omega).
\end{equation}
Then, the variational problem which is equivalent to the boundary value problem for \eqref{6.5} with the aforementioned boundary conditions is given by
\begin{equation}\label{variational prob}
B(v,w)=({\rm div}\mathfrak{T}\{\mathcal{I}\omega\},w)+(f,w)_\rho\quad\text{for all}\ w\in K(\Omega).
\end{equation}
Now, we split the bilinear form $B$ into two parts upon splitting $\mathcal{I}$ into its real and imaginary parts,
\begin{equation}\label{mathcal-I}
\mathcal{I}=K(\mu^2\breve{I}+K^2)^{-1}\breve{C}+i\{-\mu(\mu^2\breve{I}+K^2)^{-1}\breve{C}\}.
\end{equation}
That is,
\begin{equation}\label{split form of B}
\begin{array}{rl}
B=&B_1+iB_2,\\
\\
B_1(v,w)=&(\mathfrak{T}\{K(\mu^2\breve{I}+K^2)^{-1}\breve{C}(e[v]\breve{I})\},e[w])+\sigma(v,w)_\rho,\\
\\
B_2(v,w)=&-\mu (\mathfrak{T}\{(\mu^2\breve{I}+K^2)^{-1}\breve{C}(e[v]\breve{I}),e[w])+\mu (v,w)_\rho.
\end{array}
\end{equation}
Here, $B_1$ and $B_2$ are both continuous symmetric bilinear forms on $K(\Omega)$.

Using that
\begin{equation}\label{coeff of e[v]}
\mathfrak{T}\{K(\mu^2\breve{I}+K^2)^{-1}\breve{C}(e[v]\breve{I})\}=\sum_{j=0}^n C_j K_j(\mu^2 I+K_j^2)^{-1}e[v],
\end{equation}
where each $K_j$ is given by $K_j=\sigma I+\eta_j^{-1} C_j$, and \eqref{cond-sigma}, it follows that $B_1$ is coercive for such $\sigma$ due to the Korn inequality. Hence, the above variational problem is uniquely solvable by the usual argument \cite[Section 9, Chapter 3]{Mizo}. 

\subsection*{Generation of contractive {$C_0$-semigroup for $L$}}

We begin with the following

\begin{lemma}\label{lem6.1}
Let $L$ be given by (\ref{6.2}). The following holds true,
\begin{equation}\label{6.8}
\|(\lambda I-L)
\begin{pmatrix}
v                \\
\psi
\end{pmatrix}
\|_H\geq \lambda\|
\begin{pmatrix}
v                \\
\psi
\end{pmatrix}
\|_H,\quad \lambda>0,\,\,(v,\psi)^{\it t}\in D(L).
\end{equation}
\end{lemma}

\begin{proof}
For $(v,\psi)^{\it t}\in D(L)$, we have
\begin{equation}\label{6.9}
\begin{aligned}
(L
\begin{pmatrix}
v                \\
\psi
\end{pmatrix},
\begin{pmatrix}
v                \\
\psi
\end{pmatrix})_H&=
(\rho^{-1}{\rm div}\mathfrak{T}\{\breve{C}\psi\}, v)_\rho+(-\breve{\eta}^{-1}\breve{C}\psi+e[v]\breve{I},\psi)_{\breve{C}}\\
&=-(\psi,\breve{C}(e[v]\breve{I}))-(\breve{\eta}^{-1}\breve{C}\psi,\psi)_{\breve{C}}+(e[v]\breve{I},\psi)_{\breve{C}},
\end{aligned}
\end{equation}
\begin{equation}\label{6.10}
\begin{aligned}
(
\begin{pmatrix}
v                \\
\psi
\end{pmatrix},L
\begin{pmatrix}
v                \\
\psi
\end{pmatrix})_H&=
(v,\rho^{-1}{\rm div}\mathfrak{T}\{\breve{C}\psi\})_\rho+(\psi,-\breve{\eta}^{-1}\breve{C}\psi+e[v]\breve{I})_{\breve{C}}\\
&=-(\breve{C}(e[v]\breve{I}),\psi)-(\breve{\eta}^{-1}\breve{C}\psi,\psi)_{\breve{C}}+(\psi, e[v]\breve{I})_{\breve{C}}.
\end{aligned}
\end{equation}
Combining \eqref{6.9} and \eqref{6.10}, we find that 
\begin{equation}\label{6.11}
\begin{aligned}
(L
\begin{pmatrix}
v                \\
\psi
\end{pmatrix},
\begin{pmatrix}
v                \\
\psi
\end{pmatrix})_H+(
\begin{pmatrix}
v                \\
\psi
\end{pmatrix},L
\begin{pmatrix}
v                \\
\psi
\end{pmatrix})_H=
-2(\breve{\eta}^{-1}\breve{C}\psi,\psi)_{\breve{C}}\leq 0.
\end{aligned}
\end{equation}
Hence, for $\lambda>0$, $(v,\psi)^{\it t}\in D(L)$,
\begin{equation}\label{6.12}
\begin{aligned}
\|(\lambda I&-L)
\begin{pmatrix}
v                \\
\psi
\end{pmatrix}
\|_H^2\\
=& \lambda^2\|
\begin{pmatrix}
v                \\
\psi
\end{pmatrix}
\|_H^2+\|L
\begin{pmatrix}
v                \\
\psi
\end{pmatrix}\|_H^2-
\lambda(L
\begin{pmatrix}
v                \\
\psi
\end{pmatrix},
\begin{pmatrix}
v                \\
\psi
\end{pmatrix})_H-\lambda(
\begin{pmatrix}
v                \\
\psi
\end{pmatrix},L
\begin{pmatrix}
v                \\
\psi
\end{pmatrix})_H\\
\geq & \lambda^2\|
\begin{pmatrix}
v                \\
\psi
\end{pmatrix}
\|_H^2.
\end{aligned}
\end{equation}
Thus we have proved estimate \eqref{6.8}.
\end{proof}

Using the aforementioned solvability of the variational problem, we get the bijectivity of the map 
$\lambda I-L: D(L)\rightarrow H$ for $\lambda$ with ${\rm Re}\lambda>0$ and, hence, the resolvent $(\lambda I-L)^{-1}$ satisfies 
\begin{equation}\label{resolvent est for L 1}
\Vert (\lambda I-L)^{-1}\Vert\le\lambda^{-1},\,\,\lambda>0,
\end{equation}
where $\Vert\cdot\Vert$ is the operator norm on $H$. Furthermore, for the resolvent set $\rho(L)$ of $L$, we have the inclusion  
\begin{equation}\label{resolvent set of L}
\{\sigma\ge-\delta_0\}\subset\rho(L)\,\,\text{for some $\delta_0>0$}.    
\end{equation} 
With these results and observation we formulate

\begin{theorem}\label{semigroup for L}
$L$ generates a $C_0$-semigroup, $e^{tL}$, of contraction on $H$ and the imaginary axis is in $\rho(L)$.
\end{theorem}

\begin{remark}
It is possible to improve estimate \eqref{6.8}. In fact, we have
\begin{equation}\label{6.23}
\|(\lambda I-L)^{-1}
\|\le (\lambda^2+\varepsilon_2)^{-1/2},\quad \lambda>\beta
\end{equation}
for some $\varepsilon_2>0$ and $\beta <0$.
\end{remark}

\begin{proof}
{\rm Recalling \eqref{6.11}, we modify estimate \eqref{6.12} as follows.
We consider 
\begin{equation}\label{LV=G}
L
\begin{pmatrix}
v                \\
\psi
\end{pmatrix}
=
\begin{pmatrix}
g_1               \\
g_2
\end{pmatrix}
,
\end{equation}
which is equivalent to
\begin{equation}\label{componentwise LV=G}
\left\{
\begin{aligned}
\rho^{-1}{\rm div}\mathfrak{T}\{\breve{C}\psi\}=g_1,\\
-\breve{\eta}^{-1}\breve{C}\psi+e[v]\breve{I}=g_2,
\end{aligned}
\right.
\end{equation}
where $g_1\in L_\rho^2(\Omega)$ and $g_2\in L^2(\Omega)$. Then, from \eqref{componentwise LV=G}, we have 
\begin{equation*}
\begin{aligned}
(\rho g_1,v)=-(\mathfrak{T}\{\breve{C}\psi\},e[v])=(\mathfrak{T}(-\breve{\eta}e[v]\breve{I}+\breve{\eta}g_2),e[v]),
\end{aligned}
\end{equation*}
which is nothing but 
\begin{equation*}
\begin{aligned}
-(|\eta|e[v],e[v])=-(\mathfrak{T}(\breve{\eta}g_2),e[v])+(\rho g_1,v).
\end{aligned}
\end{equation*}
Hence, we have
\begin{equation}\label{6.26}
\begin{aligned}
\varepsilon_1(\|v\|+\|\nabla v\|)\leq \|g_1\|_{\rho} +\|g_2\|
\end{aligned}
\end{equation}
for some $\varepsilon_1>0$, where $\Vert\cdot\Vert$ denotes the $L^2(\Omega)$ norm.
From this estimate and the second equation of \eqref{componentwise LV=G}, we conclude that
\begin{equation}\label{6.27}
\begin{aligned}
2\varepsilon_2\Vert (v,\psi)^{\it t}\Vert_H^2\leq\Vert(g_1,g_2)^{\it t}\Vert_H^2
\end{aligned}
\end{equation}
for some $\varepsilon_2>0$.
Hence, we have obtained
\begin{equation}\label{7.11}
\|L
\begin{pmatrix}
v                \\
\psi
\end{pmatrix}
\|_H^2\geq 2\varepsilon_2 \|
\begin{pmatrix}
v               \\
\psi
\end{pmatrix}\|_H^2
.
\end{equation}

Now, we choose a constant $\beta<0$ such that $2\lambda(\breve{\eta}^{-1}\breve{C}\psi,\psi)+ \varepsilon_2\|\psi\|^2\geq 0$ 
for $\lambda>\beta$.
Then, using \eqref{6.11},
we have for $\lambda>\beta$ that
\begin{equation}\label{7.12}
\begin{aligned}
\|(\lambda I-&L)
\begin{pmatrix}
v                \\
\psi
\end{pmatrix}
\|_H^2\\
=& \lambda^2\|
\begin{pmatrix}
v                \\
\psi
\end{pmatrix}
\|_H^2+\|L
\begin{pmatrix}
v                \\
\psi
\end{pmatrix}\|_H^2-
\lambda(L
\begin{pmatrix}
v                \\
\psi
\end{pmatrix},
\begin{pmatrix}
v                \\
\psi
\end{pmatrix})_H-\lambda(
\begin{pmatrix}
v                \\
\psi
\end{pmatrix},L
\begin{pmatrix}
v                \\
\psi
\end{pmatrix})_H\\
\geq & (\lambda^2+\varepsilon_2)\|
\begin{pmatrix}
v                \\
\psi
\end{pmatrix}
\|_H^2.
\end{aligned}
\end{equation}
We have proven estimate \eqref{6.23}.}
\end{proof}

Next, we give an estimate of $(\lambda I-L)^{-1}$ for negative $\lambda$ in
\begin{lemma}\label{lem6.4}
There exists a positive constant $l$ such that
\begin{equation}\label{6.33}
\|(\lambda I-L)
\begin{pmatrix}
v                \\
\psi
\end{pmatrix}
\|_H\geq (-\lambda-l)\|
\begin{pmatrix}
v                \\
\psi
\end{pmatrix}
\|_H,\quad -\lambda>l,\,\,(v,\psi)^{\it t}\in D(L).
\end{equation}
\end{lemma}

\begin{proof}
From \eqref{6.11} and \eqref{6.12}, we have that if $\lambda \leq -l_2$ for some positive constant $l_2$, then
\begin{equation}\label{6.34}
\begin{aligned}
\|(\lambda I-&L)
\begin{pmatrix}
v                \\
\psi
\end{pmatrix}
\|_H^2\\
=& \lambda^2\|
\begin{pmatrix}
v                \\
\psi
\end{pmatrix}
\|_H^2+\|L
\begin{pmatrix}
v                \\
\psi
\end{pmatrix}\|_H^2+2(\breve{\eta}^{-1}\breve{C}\psi,\psi)_{\breve{C}}\\
\geq & (\lambda^2-2|\lambda|M)\|
\begin{pmatrix}
v                \\
\psi
\end{pmatrix}
\|_H^2,
\end{aligned}
\end{equation}
where $M>l_2$ is a positive constant.
Because
\begin{equation}\label{6.35}
\begin{aligned}
\lambda^2-2|\lambda|M&=\lambda^2-4|\lambda|M+4M^2+2|\lambda|M-4M^2\\
&=(|\lambda|-2M)^2+2M(|\lambda|-2M),
\end{aligned}
\end{equation}
we find that for $-\lambda>l$
\begin{equation}\label{6.36}
\begin{aligned}
\|(\lambda I-L)
\begin{pmatrix}
v                \\
\psi
\end{pmatrix}
\|_H^2
\geq (|\lambda|-l)^2\|
\begin{pmatrix}
v                \\
\psi
\end{pmatrix}
\|_H^2,
\end{aligned}
\end{equation}
where $l=2M$.
\end{proof}

Once again, from the aforementioned solvability of the variational problem, we have the following resolvent estimate 
\begin{equation}\label{resolvent estimate L 2}
\Vert (\lambda I-L)^{-1}\Vert\le(|\lambda|-l)^{-1},\,\,-\lambda>l,
\end{equation}
where $\Vert\cdot\Vert$ is the operator norm on $H$.

\medskip
Therefore, combining the two resolvent estimates \eqref{resolvent est for L 1} and \eqref{resolvent estimate L 2}, we arrive at the following

\begin{theorem}\label{generation of group for L}
$L$ generates a $C_0$-group on $H$.    
\end{theorem}

\paragraph{Abstract Cauchy problem} We obtain the unique solvability of the abstract Cauchy problem for the reduced system as we had obtained this for the full AD system. In Theorem~\ref{inhomog CP}, $V$ becomes $(v,\psi)^t$, $A_Z$ simply needs to be replaced by $L$ and $W_Z$ by $H$.

\section{Exponential decay property of solutions of the reduced system}

In this section, we prove that any solution $(v,\psi)$ of \eqref{6.1} whose initial data satisfying the compatibility condition of order 2 decays exponentially as $t\rightarrow\infty$.
To this end, we define the energy
\begin{equation}\label{energy norm}
E(v,\psi):=\frac{1}{2}\|v\|^2_{\rho}+\frac{1}{2}\sum_{j=1}^n (C_j\psi_j,\psi_j).
\end{equation} 

\begin{lemma}\label{K_lem7.1}
The energy defined in \eqref{energy norm} satisfies
\begin{equation}\label{K_7.5}
\begin{aligned}
\frac{d}{dt}E(v,\psi)=-\sum_{j=1}^n\eta_j\|e[v]-\partial_t\psi_j\|^2.
\end{aligned}
\end{equation}
\end{lemma}

\begin{proof}
The first equation of \eqref{6.1} implies that
\begin{equation}\label{K_7.6}
\begin{aligned}
(\rho\partial_tv,v)=({\rm div}\sum_{j=1}^n C_j\psi_j,v)=-\sum_{j=1}^n(C_j\psi_j,e[v]).
\end{aligned}
\end{equation}
The second equation of \eqref{6.1} gives
\begin{equation}\label{K_7.7}
\begin{aligned}
\eta_j(\partial_t\psi_j,\partial_t\psi_j)=(\eta_je[v]-C_j\psi,\partial_t\psi_j)=(\eta_je[v],\partial_t\psi_j)-(C_j\psi,\partial_t\psi_j).
\end{aligned}
\end{equation}
A straightforward computation, using \eqref{K_7.6} and \eqref{K_7.7}, yields
\begin{equation}\label{K_7.8}
\begin{aligned}
\frac{d}{dt}E(v,\psi)
=& \sum_{j=1}^n(C_j\psi,\partial_t\psi_j)+(\rho\partial_tv,v)\\
=&-\sum_{j=1}^n\eta_j\|\partial_t\psi_j\|^2+ \sum_{j=1}^n(\eta_je[v],\partial_t\psi_j)-\sum_{j=1}^n(C_j\psi_j,e[v])\\
=&-\sum_{j=1}^n\eta_j\|\partial_t\psi_j\|^2+ \sum_{j=1}^n2\eta_j(e[v],\partial_t\psi_j)-\sum_{j=1}^n\eta_j\|e[v]\|^2\\
=&-\sum_{j=1}^n\eta_j\|e[v]-\partial_t\psi_j\|^2,
\end{aligned}
\end{equation}
which is the statement of the lemma.
\end{proof}

We now differentiate equation \eqref{6.1} in $t$ to obtain the following system for $(\partial_t v,\partial_t \psi)$.
\begin{equation}\label{K_7.9}
\left\{
\begin{array}{ll}
\partial_t^2 v=\rho^{-1}{\rm div}\mathfrak{T}\{\breve{C}\partial_t \psi\},\\
\partial_t^2 \psi=-\breve{\eta}^{-1}\breve{C}\psi+e[\partial_t v]\breve{I},\\[0.25cm]
\partial_t v=0\quad\text{on}\quad\Gamma_D, \quad
(\mathfrak{T}\{\breve{C}\partial_t \psi\})\nu=0\quad\text{on}\quad \Gamma_N,\\[0.25cm]
(\partial_t v,\partial_t \psi)|_{t=0}\ \text{is obtained by using the first and the second equations of \eqref{6.1}.}
\end{array}
\right.
\end{equation}
The associated energy is
$$ E(\partial_tv,\partial_t\psi):=\frac{1}{2}\|\partial_tv\|^2_{\rho}+\frac{1}{2}\sum_{j=1}^n (C_i\partial_t\psi_j,\partial_t\psi_j),$$
which satisfies
\begin{equation}\label{K_7.10}
\frac{d}{dt}E(\partial_tv,\partial_t\psi)=-\sum_{j=1}^n\eta_j\|e[\partial_tv]-\partial^2_t\psi_j\|^2.
\end{equation}
in analogy to the statement in Lemma~\ref{K_lem7.1}. We then define a higher energy, $\bar E(v,\psi)$, as
\begin{equation}\label{higher energy}
    \bar E(v,\psi)=E(v,\psi)+E(\partial_t v,\partial_t \psi).
\end{equation}

For simplicity of notation, we write 
$$\Vert\tilde v\Vert^2_{\rho}:=\Vert v\Vert^2_{\rho}+\Vert\partial_t v\Vert^2_{\rho},\quad\|\tilde{\psi}\|^2=\|\psi\|^2+\|\partial_t\psi\|^2.$$
On the one hand, from the second equation of \eqref{6.1} and the Korn inequality, we have
\begin{equation}\label{K_7.14}
\begin{aligned}
a_1(\|\tilde v \|^2_{\rho}+\|\tilde{\psi}\|^2)\leq \bar{E}(v,\psi) \leq b_1(\|\tilde v\|^2_{\rho}+\|\tilde{\psi}\|^2)
\end{aligned}
\end{equation}
for some positive constants $a_1$, $b_1$ with $a_1<1<b_1$.
On the other hand, using \eqref{K_7.5} and \eqref{K_7.10}, we obtain from the second equation of \eqref{6.1},
\begin{equation}\label{K_7.15}
\begin{aligned}
\frac{d}{dt}\bar{E}(v,\psi)&=-\sum_{i=1}^n\eta_i\|e[v]-\partial_t\psi_i\|^2-\sum_{i=1}^n\eta_i\|e[\partial_tv]-\partial^2_t\psi_i\|^2\\
&\leq -a_2 \|\tilde{\psi}\|^2,
\end{aligned}
\end{equation}
for some positive constant $a_2<1$. Comparing \eqref{K_7.14} and \eqref{K_7.15}, we need to amend $\bar{E}(v,\psi)$ through adding a function $f_E$ so that $\frac{d}{dt}f_E$ has a contribution $-\|\partial_tv\|^2_{\color{blue}{\rho}}$. We define such an $f_E$ by 
$$ f_E=(\sum_{j=1}^n C_j\psi_j,e[v]).$$
Using \eqref{6.1}, a direct computation yields
\begin{equation}\label{7.16}
\begin{aligned}
\frac{d}{dt}f_E&=(\sum_{j=1}^n C_j\partial_t\psi_j,e[v])+(\sum_{j=1}^n C_j\psi_j,e[\partial_tv])\\
&=(\sum_{j=1}^n C_j\partial_t\psi_j,e[v])-({\rm div}\sum_{j=1}^n C_j\psi_j,\partial_tv)\\[0.25cm]
&\leq b_2 \|\tilde{\psi}\|^2-\|\partial_tv\|^2_{\rho}
\end{aligned}
\end{equation}
for some positive constant $b_2>1$. Using the second equation of \eqref{6.1}, the Korn inequality and \eqref{K_7.14}, we get
\begin{equation}\label{K_7.17}
\begin{aligned}
|f_E|\leq b_3\|\tilde{\psi}\|^2\leq \frac{b_3}{a_1}\bar{E}(v,\psi)
\end{aligned}
\end{equation}
for some positive constant $b_3>1$. Based on these estimates, we define the amended energy as 
\begin{equation}\label{amended energy}
\tilde{E}(v,\psi):=\bar{E}(v,\psi)+\frac{a_1a_2}{2b_2b_3}f_E.
\end{equation}
Here, we note that $\{a_1,a_2\}$ and $\{b_1,b_2,b_3\}$ can be taken arbitrarily small and large, respectively. Hence, $\tilde{E}(v,\psi)$ is equivalent in the sense of norms to $\Vert\tilde v\Vert^2+\Vert\tilde\psi\Vert^2$ 

Combining \eqref{K_7.14} to \eqref{K_7.17}, leads to the estimate for the time derivative,
\begin{equation}\label{K_7.18}
\begin{aligned}
\frac{d}{dt}\tilde{E}(v,\psi)
&\leq-a_2 \|\tilde{\psi}\|^2+\frac{a_1a_2}{2b_3} \|\tilde{\psi}\|^2-\frac{a_1a_2}{2b_2b_3}\|\partial_tv\|^2_{\rho}\\
&\leq -\frac{a_1a_2}{2b_2b_3}(\|\partial_tv\|^2_{\rho}+\|\tilde{\psi}\|^2)\\
&\leq -\frac{a_1a_2}{2b_1b_2b_3}\bar{E}(v,\psi)\\
&\leq -a_4\tilde{E}(v,\psi),
\end{aligned}
\end{equation}
where $a_4=(a_1a_2)/(3b_1b_2b_3)$.  Equation \eqref{K_7.18} implies the following exponential decay of solution $(v,\psi)$ of \eqref{6.1},
\begin{equation}\label{K_7.19}
\begin{aligned}
a_1(\|\tilde v\|^2_{\rho}+\|\tilde\psi\|^2)\leq\bar{E}(v,\psi)\leq 2\tilde{E}(v,\psi)\leq 2\tilde{E}(v(0),\psi(0))e^{-a_4t}.
\end{aligned}
\end{equation}
This proves
 

\begin{theorem}\label{exponential decay}
Let $(v(t),\psi(t))^t \in C^2([0,\infty); H^1(\Omega))\times C^2(([0,\infty); L^2(\Omega))$ be the solution of \eqref{6.1} satisying the compatibility condition of order 2. Then, there exists a constant $a_4>0$ independent of the initial values such that $(v(t),\psi(t))^t$ is exponentially decaying of order $O(e^{-a_4 t})$ as $t\rightarrow\infty$ with respect to the square root of the higher energy \eqref{higher energy}. 
\end{theorem}

\subsection*{Limiting Amplitude Principle}

In this section, we provide a proof of the limiting amplitude principle for the reduced system. To begin with, we consider the following initial boundary value problem for the reduced system:
\begin{equation}\label{lamp 1}
\left\{
\begin{array}{ll}
\partial_tv=\rho^{-1}{\rm div}\mathfrak{T}\{\breve{C}\psi\},\\
\partial_t\psi=-\breve{\eta}^{-1}\breve{C}\psi+e[v]\breve{I},\\[0.25cm]
v=e^{i\kappa t}\,\chi \tilde f\big|_{\Gamma_D}\quad\text{on}\quad\Gamma_D,\quad
(\mathfrak{T}\{\breve{C}\psi\})\nu=0 \quad \text{on}\quad \Gamma_N,\\[0.25cm]
(v,\psi)=(0,0)\quad\text{at $t=0$}.
\end{array}
\right.
\end{equation}
where $\kappa>0$ is a fixed angular frequency, $\tilde f\in H^1(\Omega)$, and $\chi=\chi(t)\in C^\infty([0,\infty))$ with the properties $\chi(t)=0$ near $t=0$, $\chi(t)=1,\,\,t\ge t_0$ for a fixed $t_0>0$. The input over $\Gamma_D$ is time harmonic after the time $t_0$. Here, note that the reason why we introduced $\chi$ is to make the boundary condition and the initial condition of \eqref{lamp 1} compatible.
 
In order to transform the boundary condition of \eqref{lamp 1} to a homogeneous one, we consider the following boundary value problem,
\begin{equation}\label{static lamp}
\left\{
\begin{array}{ll}
i\kappa\tilde v_0-\rho^{-1}{\rm div}\mathfrak{T}\{\breve{C}\tilde\psi_0\}=0,\\
i\kappa\tilde\psi_0+\breve\eta^{-1}\breve{C}\tilde\psi_0-e[\tilde v_0]\breve{I}=0,\\[0.25cm]
\tilde v_0=\tilde f\big|_{\Gamma_D}\quad\text{on
$\Gamma_D$},\quad(\mathfrak{T}\{\breve{C}\tilde\psi_0\})\nu=0\quad\text{on $\Gamma_N$}.
\end{array}
\right.
\end{equation}
The unique solvability of this problem can be shown as follows. Due to the fact $\tilde f\in H^1(\Omega)$, it can be reduced to that of the boundary value problem for \eqref{6.5} with $\lambda = i \kappa$, that is, by transforming the inhomogeneous boundary condition on $\Gamma_D$ to the homogeneous one. We use that \eqref{6.5} is uniquely solvable even in the case that $f$ is in the dual space of $K(\Omega)$ (cf.~\eqref{resolvent set of L}). Hence, there exists a unique solution $(\tilde v_0,\tilde\psi_0) \in H^1(\Omega)\times L^2(\Omega)$ of \eqref{static lamp}.

We define $(v_0,\psi_0)$ by 
\begin{equation}\label{v_0, psi_0}
(v_0,\psi_0)=\chi(t)(\tilde v_0,\tilde\psi_0)
\end{equation}
and seek a solution $(v,\psi)$ of \eqref{lamp 1} of the form,
\begin{equation}\label{form of sol}
v=e^{i\kappa t} v_0+\tilde v,\,\,\psi=e^{i\kappa t}\psi_0+\tilde\psi.
\end{equation}
Then $(\tilde v,\tilde\psi)$ has to satisfy the following initial boundary value problem,
\begin{equation}\label{tilde(v,psi) eq}
\left\{
\begin{array}{ll}
\partial_t\tilde v-\rho^{-1}{\rm div}\mathfrak{T}\{\breve{C}\tilde\psi\}=-e^{i\kappa t}\dot\chi(t)\tilde v_0\,=:\dot \chi(t)\tilde F_1(\kappa),\\
\partial_t\tilde\psi+\breve{\eta}^{-1}\breve{C}\tilde\psi-e[\tilde v]\breve{I}=-e^{i\kappa t}\dot\chi(t)\tilde\psi_0=:\dot\chi(t) \tilde F_2(\kappa),\\[0.25cm]
\tilde v=0\quad\text{on $\Gamma_D$},\quad(\mathfrak{T}\{\breve{C}\tilde\psi\})\nu=0\quad\text{on $\Gamma_N$},\\[0.25cm]
(\tilde v,\tilde\psi)=0\quad\text{at $t=0$}.
\end{array}
\right.
\end{equation}
Here, $\dot\chi(t):=\frac{d\chi}{dt}(t)$. As $$F(t):=(\dot \chi(t)\tilde F_1(\kappa), \dot\chi(t)\tilde F_2(\kappa))^t \in C^\infty([0,\infty);H)$$ is $0$ near $t=0$, any order of the compatibility condition for  \eqref{tilde(v,psi) eq} is satisfied. Using the semigroup, $e^{t L}$, solutions $\tilde V:=(\tilde v,\tilde\psi)^t$ take the form
\begin{equation}
\tilde V(t)=\int_0^t e^{(t-s)L}\,e^{i\kappa s}\dot\chi(s)\,ds\tilde F(\kappa),
\end{equation}
where $\tilde F(\kappa)=(\tilde F_1(\kappa),\tilde F_2(\kappa))^{\text{\it t}}$.
By a straightforward computation, we obtain
\begin{equation}\label{computation of tilde V}
\tilde V(t)=e^{tL}\left(I+\int_0^{t_0}(i\kappa I-L)\,e^{(i\kappa I-L)s}(1-\chi(s))\,ds\right)\tilde F(\kappa)=e^{tL}\left(I+O(1)\right)\,\tilde F(\kappa),
\end{equation}
where $O(1)$ denotes a term which is uniformly bounded in time. Hence, by the exponential decay of solutions for the reduced system, we have
\begin{equation}\label{lamp form}
\Vert\tilde V(t)\Vert=O(e^{-a_4 t})\,\,\text{as $t\rightarrow\infty$}.
\end{equation}
We have proved

\begin{theorem}
Let $(v(t),\psi(t))^t \in C^2([0,\infty); H^1(\Omega))\times C^2(([0,\infty); L^2(\Omega))$ be the solution of \eqref{6.1} satisfying the compatibility condition of order 2. Then, there exists a constant $a_4>0$ independent of the initial data such that $(v(t),\psi(t))^t$ converges to $e^{i\kappa t} (\tilde v_0,\tilde\psi_0)^t$ exponentially fast of order $O(e^{-a_4 t})$ as $t\rightarrow\infty$ with respect to the amended energy \eqref{amended energy}. Here $\tilde (\tilde v_0,\tilde\psi_0)^t \in H^1(\Omega)\times L^2(\Omega)$ is the unique solution of \eqref{static lamp}.
\end{theorem}

\bibliographystyle{siamplain}
\bibliography{references}

\begin{thebibliography}{10}

\bibitem{C}
{\sc R.~Christensen}, {\em Theory of Viscoelasticity}, Academic Press, New
  York, 1982.

\bibitem{Dafermos}
{\sc C.~Dafermos}, {\em An abstract {V}olterra equation with application to
  linear viscoelasticity}, Journal of Differential Equations, 7 (1970),
  pp.~554--569.

\bibitem{IDUD_2023}
{\sc M.~de~Hoop, C.-L. Lin, and G.~Nakamura}, {\em Decay of solutions to mixed
  boundary value problem for anisotropic viscoelastic systems}, ArXiv,  (2023).

\bibitem{Devaut}
{\sc G.~Devaut and J.~Lions}, {\em Inequalities in Mechanics and Physics},
  Springer-Verlag, Berlin, 1976.

\bibitem{Higashimori}
{\sc N.~Higashimori}, {\em Identification of viscoelastic properties by
  magnetic resonance elastography}, J. Phys. Conf. Ser., 73 (2007), p.~012009,
  \url{https://doi.org/10.1088/1742-6596/73/1/012009}.

\bibitem{Ikawa}
{\sc M.~Ikawa}, {\em Mixed problems for hyperbolic equations of second order},
  J. Math. Soc. Japan, 20 (1969), pp.~580--608.

\bibitem{Jiang}
{\sc Y.~Jiang, H.~Fujiwara, and G.~Nakamura}, {\em Approximate steady state
  model for magnetic resonance elastography}, SIAM J. Appl. Math., 71 (2011),
  pp.~1965--1989.

\bibitem{KNTY2019}
{\sc M.~Kimura, H.~Notsu, Y.~Tanaka, and H.~Yamamoto}, {\em The gradient flow
  structure of an extended maxwell viscoelastic model and a
  structure-preserving finite element scheme}, Journal of Scientific Computing,
   (2019), pp.~1111--1131.

\bibitem{Lakes}
{\sc R.~Lakes}, {\em Viscoelastic Materials}, Cambridge University Press, New
  York, 2002.

\bibitem{JTLi2023}
{\sc J.-T. Li}, {\em Mathematical analysis of the zener-type viscoelastic wave
  equation}, Master Thesis, Graduate School of Natural Science and Technology,
  Kanazawa University,  (2023 January), p.~29 pages.

\bibitem{Mizo}
{\sc S.~Mizohata}, {\em The Theory of Partial Differential Equations},
  Cambridge University Press, 1979.

\bibitem{Muthupillai}
{\sc R.~Muthupillai, D.~Lomas, P.~Rossman, J.~Greenleaf, A.~Manduca, and
  R.~Ehman}, {\em Magnetic resonance elastography by direct visualization of
  propagating acoustic strain waves}, Science, 269 (1995), pp.~1854--1857.

\bibitem{Papazoglou}
{\sc S.~Papazoglou, S.~Hirsch, J.~Braum, and I.~Sack}, {\em Multifrequency
  inversion in magnetic resonance elastography}, Physics in Medicine and
  Biology, 57 (2012), pp.~2329--2346.

\bibitem{Pazy}
{\sc A.~Pazy}, {\em Semigroups of Linear Operators and Applications to Partial
  Differential Equations}, Springer, New York, 1983.

\bibitem{Peltier_1974}
{\sc W.~Peltier}, {\em The impulse response of a {M}axwell {Earth}}, Rev.
  Geophysics, 12 (1974), pp.~649--669.

\bibitem{Potier}
{\sc M.~Potier-Ferry}, {\em On the multidimensional foundation of elastic
  stability theory, i}, Arch. Ration. Mech. Analy., 78 (1972), pp.~55--72.

\bibitem{S}
{\sc R.~Showalter}, {\em Hilbert Space Methods for Partial Differential
  Equations}, Dover Publications, 2010.

\bibitem{Yamamoto2019}
{\sc H.~Yamamoto}, {\em Energy gradient flow structure of the zener
  viscoelastic model and its finite element analysis}, Master Thesis, Graduate
  School of Natural Science and Technology, Kanazawa University,  (2019), p.~36
  pages (in Japanese).

\end{thebibliography}

\end{document}